\documentclass[12pt]{amsart}

\usepackage[T1]{fontenc}
\usepackage[utf8]{inputenc}
\usepackage[polish]{babel}
\AtBeginDocument{}
\AtBeginDocument{}

\usepackage{amsfonts,amssymb,amsmath,amsthm}
\usepackage{url}
\usepackage[colorlinks=true, urlcolor=black, citecolor=black, linkcolor=black, hyperfootnotes=true]{hyperref}
\usepackage{bbm}
\usepackage{bm}
\usepackage{anysize}
\usepackage{cleveref}
\usepackage{dutchcal}

\usepackage{tikz-cd}
\usetikzlibrary{decorations.markings, arrows.meta}
\tikzset{
	marrow/.style={decoration={markings,mark=at position 0.5 with {\arrow{#1}}}, postaction=decorate}
}

\numberwithin{equation}{section}

\theoremstyle{plain}
\newtheorem{theorem}{Theorem}[section]
\newtheorem{lemma}[theorem]{Lemma}
\newtheorem{corollary}[theorem]{\bf Corollary}
\newtheorem{remark}[theorem]{\bf Remark}

\newtheorem{proposition}[theorem]{\bf Proposition}

\theoremstyle{definition}

\newcommand{\nil}{\mathrel{\vphantom{>}\text{\mathsurround=0pt\ooalign{$>$\cr$-$\cr}}}}

\newcommand{\sqni}{\mathrel{\vphantom{\sqsupset}\text{\mathsurround=0pt\ooalign{$\sqsupset$\cr$-$\cr}}}}

\newcommand{\sqin}{\mathrel{\vphantom{\sqsubset}\text{\mathsurround=0pt\ooalign{$\sqsubset$\cr$-$\cr}}}}

\newcommand \Homeo{ \mbox{Homeo}}

\newcommand \Id{ \mbox{Id}}
\newcommand \Sp{{\mathsf{S}}}

\newcommand \mesh{{\rm{ mesh}}}

\newcommand \sqsp{\sqsupset}
\newcommand \sqsb{\sqsubset}
\newcommand \la{\leftarrow}
\newcommand \ra{\rightarrow}
\newcommand \tla{\twoheadleftarrow}

\newcommand{\cl}{\mbox{cl}}

\newcommand \mor{{\rm{mor}}}
\newcommand \ob{{\rm{ob}}}

\newcommand \rt{\mathbf{0}}

\newcommand \diam{{\rm{ diam}}}

\newcommand \fra{Fra\"{i}ss\'{e}}
\newcommand \uhr{\upharpoonright}

\newcommand \AAb{\mathbf{A}}
\newcommand \CCb{\mathbf{C}}

\newcommand \GGb{\mathbf{G}}
\newcommand \GGRb{\mathbf{Gr}}
\newcommand \GGRFb{\mathbf{Gr_F}}
\newcommand \IIb{\mathbf{I}}
\newcommand \JJb{\mathbf{J}}
\newcommand \KKb{\mathbf{K}}
\newcommand \LLb{\mathbf{L}}

\newcommand \SSb{\mathbf{S}}

\newcommand \AAub{\underline{\mathbf{A}}}
\newcommand \CCub{\underline{\mathbf{C}}}
\newcommand \DDub{\underline{\mathbf{D}}}

\newcommand \GGRub{\underline{\mathbf{Gr}}}

\newcommand \IIub{\underline{\mathbf{I}}}
\newcommand \JJub{\underline{\mathbf{J}}}
\newcommand \KKub{\underline{\mathbf{K}}}
\newcommand \LLub{\underline{\mathbf{L}}}
\newcommand \MMub{\underline{\mathbf{M}}}
\newcommand \SSub{\underline{\mathbf{S}}}

\newcommand \KK{\mathbb{K}}

\newcommand \OO{\mathcal{O}}
\newcommand \oo{\mathcal{o}}

\newcommand \sss{\mathcal{s}}
\newcommand \ttt{\mathcal{t}}

\newcommand \NN{\mathbb{N}}

\newcommand \QQ{\mathbb{Q}}

\newcommand \ZZ{\mathbb{Z}}

\newcommand \It{\mathtt{I}}

\author[M. Malicki]{Maciej Malicki}
\address{Faculty of Mathematics, Informatics and Mechanics of the University of Warsaw, ul. Banacha 2, Warsaw, Poland}
\email{mmalicki@mimuw.edu.pl}

\thanks{Research was partially supported by the National Science Centre, Poland under the Weave-UNISONO call in the Weave programme [grant no 2021/03/Y/ST1/00072].}

\begin{document}
	
\title[Rokhlin properties]{Rokhlin properties in homeomorphism groups of  compact metric spaces}
	
\begin{abstract}
We develop a general framework for studying generic homeomorphisms of compact metric spaces using combinatorial amalgamation methods inspired by Fraïssé theory. Our approach combines Rosendal’s criterion for comeager conjugacy classes with combinatorial codings of compact spaces arising from the Bartoš-Bice-Vignati duality. We introduce a new amalgamation property, called weak minimal amalgamation, formulated in suitable paracategories encoding local dynamical data of homeomorphisms. This yields characterizations of the Rokhlin and the strong Rokhlin properties, i.e., the existence of dense and comeager conjugacy classes in homeomorphism groups of compact metric spaces.
As applications, we reprove Hjorth’s theorem that $\Homeo_+([0,1])$ admits a generic homeomorphism, and establish the existence of generic homeomorphisms for the Cantor fan and the Lelek fan. Moreover, we show that generic homeomorphisms of these spaces have no Li–Yorke pairs, and therefore generically have zero topological entropy. More broadly, the paper develops new connections between combinatorial amalgamation properties and generic phenomena in topological group theory and topological dynamics.	
\end{abstract}
	
	\maketitle
	
	\setcounter{tocdepth}{2}
	
\section{Introduction}

The paper is a contribution to the study of isomorphism-invariant properties shared by ``typical'' homeomorphisms of compact metric spaces. Here, homeomorphisms are viewed as elements of the homeomorphism group $\Homeo(X)$ endowed with the Polish group topology induced by the supremum metric. Two homeomorphisms $f,g \in \Homeo(X)$ are \emph{isomorphic} if there exists a homeomorphism $h \in \Homeo(X)$ such that $$f \circ h=h \circ g,$$ i.e., if $f$ and $g$ belong to the same conjugacy class of $\Homeo(X)$. Since $\Homeo(X)$ is Baire, the notion of a ``typical'' homeomorphism can be formalized using Baire category methods: a property $P(x)$ is said to be \emph{generic} if the set
\[
\{x \in \Homeo(X): P(x)\}
\]
is comeager in $\Homeo(X)$. Recall that comeager sets are topological analogues of full measure sets; in particular, countable intersections of comeager sets are again comeager, and the complement of a comeager set is not comeager.

The search for properties of typical homeomorphisms becomes more tractable when $\Homeo(X)$ admits a \emph{generic homeomorphism}, i.e., an element whose conjugacy class is comeager. Indeed, every isomorphism-invariant property of such a homeomorphism automatically holds generically. In topological dynamics, the existence of a comeager conjugacy class is sometimes called the \emph{strong Rokhlin property}, while the existence of a dense conjugacy class is referred to as the \emph{Rokhlin property}. Rokhlin properties as a tool in investigating generic properties of homeomorphisms already appear in the work of Glasner--Weiss \cite{GlWe}, and later Bernardes--Darji \cite{BeDa}. In particular, Glasner--Weiss used the Rokhlin property to show that zero topological entropy is generic for the Cantor space, and that infinite topological entropy is generic for the Hilbert cube; Bernardes-Darji used the strong Rokhlin property to reprove the former result. On the other hand, the fact that the homeomorphism group of the pseudo-arc has the Rokhlin property, proved in Bice-Malicki \cite{BiMa} implies that there can be no useful continuous invariant for homeomorphisms of the pseudo-arc analogous to the degree of a homeomorphism of the Knaster continuum established by Dębski \cite{De}.  The existence of generic homeomorphisms is also closely related to the structure of Polish groups. For example, it plays a crucial role in Rosendal--Solecki's proof of automatic continuity of $\Homeo([0,1])$ \cite{RoSo}. More generally, Rosendal conjectured that every Polish group with a comeager conjugacy class has the automatic continuity property.

The present paper proposes a general framework for establishing the existence of generic homeomorphisms of compact metric spaces, together with some applications. The methods originate in model theory, more specifically in Fra\"iss\'e theory (see, e.g., \cite{Fr}). In the classical setting, one starts with a category $\KKb$ of finite structures (with embeddings as morphisms) that has \emph{amalgamation}, i.e., for any $A,B_0,B_1 \in \KKb$, and morphisms $\alpha_0: A \to B_0$, $\alpha_1: A \to B_1$ there exist $C \in \KKb$, and morphisms $\alpha'_0: B_0 \to C$, $\alpha'_1: B_1 \to C$ such that $$\alpha'_0 \circ \alpha_0=\alpha'_1 \circ \alpha_1,$$ and constructs from it a canonical countable limit object, the Fra\"iss\'e limit of $\KKb$. A basic example is the ordered set $(\QQ,\leq)$ of the rational numbers, obtained as the Fra\"iss\'e limit of the category of finite linear orders.

Since the original work of Fra\"iss\'e, these ideas have proved remarkably robust and have by now found applications far beyond classical model theory. Over the last two decades, various analogues of Fra\"iss\'e limits have appeared in topology, functional analysis, operator algebras, and topological dynamics. Examples include the pseudo-arc, the Gurarii space, the Poulsen simplex, the Jiang--Su algebra, and the hyperfinite $\mbox{II}_1$ factor (see, e.g., \cite{IrSo}, \cite{Lu}, and \cite{EaFa} for more details). These constructions reveal a common paradigm: highly homogeneous objects often arise as limits of finite or finitely generated approximations satisfying suitable amalgamation principles.

The connection between Fra\"iss\'e theory and generic automorphisms was first explored by Ivanov \cite{Iv}, then, among others, by Kechris--Rosendal \cite{KeRo}, where a variant of amalgamation, called weak amalgamation, was used to characterize the existence of generic automorphisms of countable structures. More recently, Bice--Malicki \cite{BiMa} developed techniques for studying generic homeomorphisms of compact metric spaces building on the framework introduced by Barto\v s--Bice--Vignati \cite{BaBiVi} and \cite{BaBiVi2}. In those works, compact metric spaces are reconstructed from combinatorial data encoded by categories of graphs, and morphisms that are not necessarily functions. More precisely, one considers sequences of finite graphs, giving rise to partially ordered sets, and associates to them a compact space called the \emph{spectrum} of the poset. This construction is analogous to Stone duality: points of the spectrum correspond to certain filters on the poset, while the topology is generated by natural basic open sets determined by its elements (see \cite[Section 3]{Ma} for a short introduction to this topic). This approach also yields a characterization, in terms of directedness as in \cite{KeRo} and \cite{BiMa}, of the existence of a dense conjugacy class.

An alternative projective approach, developed in Irwin--Solecki \cite{IrSo}, constructs compact spaces as quotients of zero-dimensional projective Fra\"iss\'e limits. Both approaches emphasize that compact spaces can be encoded by finite combinatorial approximations together with suitable amalgamation properties. Related finite-combinatorial methods have recently been used by Basso--Codenotti--Vaccaro \cite{BaCoVa} to study generic-orbit problems for the action of homeomorphism groups of Peano continua on spaces of maximal chains.

In \cite{BiMa}, generic homeomorphisms of compact metric spaces are constructed using abstract Banach--Mazur games and certain relations on posets, called refiners, which encode homeomorphisms of the associated spectra. Here we take a different approach, based on Rosendal's criterion for the existence of a comeager orbit of a continuous action of a Polish group, which allows us to dispense with refiners. We obtain a direct characterization of the existence of a generic homeomorphism in terms of a new amalgamation property -- called \emph{weak minimal amalgamation} -- in a suitable paracategory coding a $\pi$-basis of $\Homeo(X)$. More generally, the criterion applies to closed subgroups of $\Homeo(X)$ and reduces the existence of a comeager conjugacy class to a finite combinatorial amalgamation problem. Weak minimal amalgamation can be viewed as a relaxation of classical Fra\"iss\'e amalgamation adapted both to local conjugacy data and to the fact that the relevant morphisms are typically relations rather than functions.

We apply this criterion first to reprove Hjorth's theorem \cite{Hj} asserting that the group $\Homeo_+([0,1])$ of increasing homeomorphisms of the arc $[0,1]$ has a generic homeomorphism, and we show that there exists a generic homeomorphism of the Cantor fan and of the Lelek fan. Barto\v{s}ov\'a--Kwiatkowska \cite{BaKw} had previously proved that the homeomorphism group of the Lelek fan has a dense conjugacy class; thus, in this case our result strengthens the Rokhlin property to the
strong Rokhlin property. We also use Barto\v{s}--Bice--Vignati duality, and Fra\"iss\'e-like techniques to study their generic dynamics.  We show that a generic homeomorphism of either fan has no Li--Yorke pairs; in particular, zero topological entropy is generic. Previously, among non-zero-dimensional compact metric spaces, generic homeomorphisms were known only for the arc and for the degree-1 homeomorphisms of the Knaster continuum, for which Iyer \cite{Iy} recently established the existence of a generic homeomorphism of degree $1$. On the other hand, to the best of our knowledge, generic properties of homeomorphisms that are related to topological dynamics are known only for the Cantor space, and the Hilbert cube.  

The methods also suggest two directions for further work. One is to study generic homeomorphisms of other compacta, such as the pseudo-arc, the Sierpiński carpet, or the Hilbert cube. Another is to relate combinatorial amalgamation properties systematically to generic dynamical properties, for example recurrence, chain transitivity, shadowing-type properties, or the structure of invariant measures. Related connections between the strong topological Rokhlin property, symbolic dynamics, and generic shadowing have recently been established by Doucha \cite{Do}; see also \cite{DoKwMaNi} for Fra\"iss\'e-theoretic methods applied to invariant measures on the Cantor space. 

The paper is organized as follows. In Section 2, we recall basic notions concerning relations and category theory. In Section 3, for a compact metric space $X$, we define a paracategory $\KKb_X$ consisting of finite reflexive graphs arising from open covers of $X$, where adjacency corresponds to non-empty intersection. From this we construct a paracategory $\KKub_X$ of directed graphs coding information about homeomorphisms of $X$. These digraphs determine a basis of neighborhoods in $\Homeo(X)$ by recording how homeomorphisms move points between members of the cover; we also consider paracategories $\KKub_G$ corresponding to subgroups $G \leq \Homeo(X)$. In Section 4, we prove characterizations of the existence of a dense and of a comeager conjugacy class in a Polish $G \leq \Homeo(X)$ in terms of directedness and weak minimal amalgamation. In Section 5, we specialize the theory to spaces arising as Fra\"iss\'e limits of graph categories in the sense of \cite{BaBiVi2}. Finally, Section 6 contains applications to the arc, the Cantor fan, and the Lelek fan.

\section{Basic notions}
\subsection{Relations}
Let $R \subseteq A \times B$, $S \subseteq B \times C$ be relations. For $X \subseteq A$, $$X^R=\{b \in B: \exists a \in X \, a R b\}.$$ The composition $R \circ S \subseteq A \times C$, and inverse $R^{-1} \subseteq B \times A$ are defined in a standard manner, i.e., $$R \circ S =\{ (a,c) \in A \times C: \exists b (a \, R \, b \, S \, c)\}, \ R^{-1}=\{(b,a) \in (B,A): aRb\}.$$ 
We will often denote the inverse of a relation by the inverse of its symbol, e.g., $\sqsb$ is the inverse of $\sqsp$.
We think of $A$ as the co-domain, and $B$ as the domain of $R$. Thus, $R$ is \emph{surjective} if for every $a \in A$ there is $b \in B$ such that $aRb$. It is \emph{co-surjective} if $R^{-1}$ is surjective. It is \emph{co-injective} if for every $a \in A$ there is $b \in B$ such that  
$$ aRb \mbox{ and } \forall a' \in A (a'Rb \rightarrow a=a').$$
Any such $b$ will be called  \emph{$R$-co-injective}. Finally, $R$ is \emph{co-bijective} if it is co-injective and co-surjective.

The relation $\pi_R \subseteq A \times B$ is defined as the largest $R' \subseteq R$ such that  
$$ (a R' b \mbox{ and }  a'R b) \rightarrow a=a'.$$
\begin{remark}
Note that $\pi_R$ is a partial function $\pi_R: B \to A$, and $R$ is co-injective iff $\pi_R$ is surjective.
\end{remark}
Let $A$, $B$ be sets. Consider binary relations defined on both $A$ and $B$ that are denoted by the same symbol $R$. We say that a relation $\sqsp \subseteq A \times B$  is \emph{$R$-preserving} if, 
for all  $b_0, b_1 \in B$,  $a_0, a_1 \in A$
$$(a_0 \sqsp b_0 \mbox{ and } a_1 \sqsp b_1 \mbox{ and } R(b_0,b_1)) \rightarrow R(a_0,a_1).$$
It is \emph{weakly $R$-preserving} if
for all $ b_0, b_1 \in B$ 
$$R(b_0,b_1) \rightarrow \exists a_0, a_1 \in A (a_0 \sqsp b_0 \mbox{ and } a_1 \sqsp b_1 \mbox{ and } R(a_0,a_1)).$$
%
It is \emph{$R$-surjective} if for all $a_0, a_1 \in A$,
 $$R(a_0,a_1) \rightarrow \exists b_0, b_1 \in B (a_0 \sqsp b_0 \mbox{ and }  a_1 \sqsp b_1 \mbox{ and }  R(b_0,b_1)).$$ 
If $R$ is a reflexive graph relation, i.e., $R$ is reflexive and symmetric, we say that $\sqsp \subseteq A \times B$ is \emph{edge-witnessing} if, for all $a_0, a_1 \in G$, 
$$R(a_0, a_1) \rightarrow \exists b \in B (a_0 \sqsp b \mbox{ and } a_1 \sqsp b).$$ 


\subsection{Categories and paracategories}

Recall that a \emph{paracategory} is defined similarly to a category, except that composition of morphisms is not required to be defined for every composable pair. In this paper, we will consider only paracategories $\KKb$ whose objects $A \in \ob(\KKb)$ are finite sets with relations, and morphisms $\sqsp \in \mor(\KKb)$ from $B$ to $A$, where $A, B \in \ob(\KKb)$, are relations $\sqsp \subseteq A \times B$. Composition of morphisms is composition of relations. For convenience, we will often write $A \in \KKb$ instead of $A \in \ob(\KKb)$. For $A,B \in \KKb$, the symbol $\KKb^A_B$ denotes the class of all morphisms in $\KKb$ from $B$ to $A$, and $\KKb_\bullet^A$ is the class of all morphisms to $A$.

A sub-paracategory $\LLb \subseteq \KKb$ is \emph{wide} if $\ob(\LLb)=\ob(\KKb)$. It is \emph{full} if $\LLb^A_B=\KKb^A_B$ for any $A,B \in \LLb$.  
For $A, B \in \KKb$, and $\sqsp \in \KKb^A_B$, we call $B$ a \emph{$\sqsp$-lift} of $A$; it is a lift of $A$ if there is such $\sqsp$, and, for a sub-paracategory $\It \subseteq \KKb$, it is an $\It$-lift if there is such $\sqsp \in \mor(\It)$.

Let $\KKb \subseteq \GGb$ be a sub-paracategory of a category $\GGb$. In particular, compositions of morphisms in $\KKb$ exist (in $\GGb$). 
We say that $\KKb$ has \emph{weak amalgamation} if for any $A \in \KKb$, there are $B \in \KKb$, and $\nil \in \KKb^A_{B}$ such that for any $C_0,C_1 \in \KKb$, and $\sqsp_0 \in \KKb^B_{C_0}$, $\sqsp_1 \in \KKb^B_{C_1}$, there are $D \in \KKb$, and $\sqsp'_0 \in \KKb^{C_0}_D$, $\sqsp'_1 \in \KKb^{C_1}_D$ such that $$\nil \circ \sqsp_0 \circ \sqsp'_0=\nil \circ \sqsp_1 \circ \sqsp'_1.$$
We refer to $\nil$ (and $B$, if $\nil$ is clear from the context) an \emph{amalgamation base} for $A$. We call $D$, and $\sqsp'_0$, $\sqsp'_0$, an \emph{amalgam}  of $\sqsp_0$, $\sqsp_1$ over $\nil$ (or of $C_0$, $C_1$  over $A$, if the morphisms are clear from the context). When the above condition is replaced with the condition $$ \nil \circ \sqsp_0 \circ \sqsp'_0 \subseteq \nil \circ \sqsp_1 \circ \sqsp'_1,$$
we say that $\KKb$ has \emph{weak lax-amalgamation}. When it is replaced with the condition that the intersection  $$ \nil \circ \sqsp_0 \circ \sqsp'_0 \cap \nil \circ \sqsp_1 \circ \sqsp'_1$$
is co-surjective as a relation in $A \times D$, then we say that $\KKb$ has \emph{weak minimal amalgamation} or, shortly, \emph{weak m-amalgamation}. 
 Also, $\KKb$ has amalgamation, lax-amalgamation or m-amalgamation, provided that in the respective conditions we can assume that $B=A$, and $\nil=\Id$. Note that if $\KKb$ has (weak) amalgamation, then it has (weak) m-amalgamation. 
 A category that is closed under isomorphism, countable up to isomorphism, directed, and has (weak) amalgamation, is called a \emph{(weak) \fra \ category}.
 
 \begin{remark}
Most basic facts about \fra \ categories are also true about \fra \ paracategories (defined analogously). In this paper, we restrict ourselves to \fra \ categories because so far we are not aware of any interesting and useful \fra \ paracategories that are not categories.
 \end{remark}
 
Let $\It \subseteq \KKb$ be a subcategory of a paracategory $\KKb$. We say that $\KKb$ is \emph{weakly $\It$-absorbing} if, for any $A \in \It $, there are $B \in \It$, and $\nil \in \It^A_{B}$ such that for any $C \in \KKb$, and $\sqsp \in \KKb^B_{C}$ there are $D \in \It$, and $\sqsp' \in \KKb^{C}_D$ such that $$\nil \circ \sqsp \circ \sqsp' \in \mor(\It).$$ It is \emph{weakly $\It$-lax-absorbing} if for any $\nil$, $\sqsp$, $\sqsp'$ as above, there is $\sqni \in \mor(\It)$ such that $$\nil \circ \sqsp \circ \sqsp' \subseteq \sqni.$$ 
It is  \emph{$\It$-(lax)-absorbing} provided that we can assume that $B=A$, and $\nil=\Id$. We call a sequence $(K_n, \succeq_n)$ in $\KKb$ a \emph{(weak) \fra \ sequence} if, for the category $\It \subseteq \KKb$ whose objects are $K_n$, $n \in \NN$, morphisms are $\succeq^m_n$, where $m \leq n$, and $$\succeq^m_m=\Id, \quad \succeq^m_n= \succeq_m \circ \succeq_{m+1} \circ \ldots \circ \succeq_n,$$ for $m<n$, the family $\{K_n\}$  is coinitial in $\ob(\KKb)$, and $\KKb$ is (weakly) $\It$-absorbing. The following fact is straighforward.

\begin{proposition}
\label{pr:FraConditions}
Let $\KKb$ be a category. The following conditions are equivalent:
\begin{enumerate}
\item $\KKb$ is a (weak) \fra \ category,
\item there exists a (weak) \fra \ sequence in $\KKb$,
\item there exists a subcategory $\It \subseteq \KKb$ such that $\KKb$ is (weakly) $\It$-absorbing.
\end{enumerate}
Moreover, if $\KKb$ ia a (weak) \fra \ category, then $(K_n,\succeq_n)$ is a (weak) \fra \ sequence iff the category $\It \subseteq \KKb$ defined by $\ob(\It)=\{K_n\}$, $\mor(\It)=\{\succeq^m_n\}$ is coinitial, and $\KKb$ is (weakly) $\It$-absorbing.	
\end{proposition}



\subsection{Graphs and digraphs} Let $\GGRb$ be the class of all finite reflexive graphs $(K,\sqcap)$, with the edge relation always denoted by $\sqcap$.  We regard $\GGRb$ as a category with morphisms $\sqsp \in \GGRb^K_L$, $K, L \in \GGRb$, defined as all relations  $\sqsp \subseteq K \times L$ that are co-bijective, and  $\sqcap$-preserving. We may also refer to $\sqcap$-preserving relations as \emph{edge-preserving}. 
The subcategory $ \GGRFb \subseteq \GGRb$ is a wide subcategory such that morphisms in $\GGRFb$ are all the morphisms $\sqsp \in \mor(\GGRb)$ that are functions $\sqsp:B \to A$.


For a paracategory $\KKb \subseteq \GGRb$ with a fixed subcategory $\It \subseteq \KKb$ (for $\KKb=\GGRb$, we always set $\It=\GGRb$), we define $\KKub$ as the paracategory of all triples $(K,\sqcap, \la)$, where  $(K,\sqcap) \in \It$, $$\la=\sqsp \circ \preceq$$ for some $K, L \in \It$, $\sqsp \in \KKb^K_L$, $\succeq \in \It^K_L$, and whose morphisms are all weakly $\la$-preserving morphisms in $\KKb$.  We will refer to objects of $\KKub$ as \emph{$\KKb$-digraphs}, or, shortly, digraphs, and, sometimes, to morphisms in $\KKub$  as \emph{weakly arrow-preserving}. We will frequently refer to $(K,\sqcap, \la)$ by $\la$, and we will denote $K$ by $K_\la$.


\begin{remark}
\label{re:FraisseDigraphs}
	Assume that $\KKb \subseteq \GGRb$ is a \fra \ category with a subcategory $\It$ defined as in Proposition \ref{pr:FraConditions} for some \fra \ sequence $(K_n,\succeq_n)$. Let $\la \in \GGRub$. Then amalgamation immediately implies that $\la \in \KKub$ iff $K_\la \in \It$, and  there are $L \in \It$,  and $\sqsp_0, \sqsp_1\in \KKb^{K_\la}_L$ such that $$\la=\sqsp_0 \circ \sqsubset_1.$$ In particular, if $\la \in \KKub$, then $\ra \in \KKub$.
	%
\end{remark}


\section{Compact metric spaces,  their groups of homeomorphisms, and induced paracategories}
The graphs $(K,\sqcap) \in \GGRb$ that are considered in this paper, are essentially finite open covers $K$ of a fixed compact metric space $X$, with the non-empty intersection relation on $K$ as the edge relation $\sqcap$. They will play the role of finite approximations of $X$. On the other hand, the arrows $\la$ in digraphs $(K,\sqcap, \la)$ will play the role of finite approximations of autohomeomorphisms of $X$.

Let $(X,d)$ be a compact metric space. The supremum metric on $\Homeo(X)$ induced by $d$ is also denoted by $d$. For a finite open cover $K$ of $X$, $\ell(K)$ denotes the Lebesgue number of $K$. It is well known that $\Homeo(X)$ with the supremum metric is a Polish group, i.e., a topological group whose topology is separable and completely metrizable,  and that a subgroup $G \leq \Homeo(X)$ is Polish iff it is closed. See \cite{Gao} for more information on basic properties of Polish spaces and Polish groups. 

Let $K, L$ be covers of $X$ such that $K$ is refined by $L$. We say that $K$ is a \emph{consolidation}  of $L$ (or that $K$ \emph{consolidates} $L$)  if every element of $K$ is a union of elements of $L$, i.e., for $p \in K$,  $$p=\bigcup \{q \in L: q \subseteq p \}.$$ 
 
For the rest of this, and the next section, we fix a compact metric space $(X,d)$, and a sequence $(K_n)$ of minimal open covers of $X$ such that
\begin{enumerate}
	\item each $K_n$ consolidates $K_{n+1}$,
	\item $\mesh(K_{n}) \to 0$.
\end{enumerate}
Such a sequence always exists by \cite[Theorem 1.34]{BaBiVi} (see \cite[Theorem 3.6]{Ma} for a short proof). As mentioned above, we regard $K_n$'s as graphs $(K_n, \sqcap) \in \GGRb$ with the non-empty intersection relation $\sqcap$.  For $f \in \Homeo(X)$, $m \leq n$, let $\sqsp_{f,m,n} \subseteq K_m \times K_n$ be defined by $$p \sqsp_{f,m,n} q \iff p \supseteq f[q],$$ $p \in K_m, q \in K_n$. 
\begin{remark}
Clearly, the relation $\sqsp_{f,m,n}$ is always $\sqcap$-preserving. Moreover, by minimality of $K_n$'s, if it co-surjective, then it is also co-injective, i.e.,  $\sqsp_{f,m,n} \in \GGRb^{K_m}_{K_n}$.
\end{remark}

\begin{proposition}
For all $f\in \Homeo(X)$, and $m \in \NN$, there is $n \geq m$ such that $\sqsp_{f,m,n} \in \GGRb^{K_m}_{K_n}$. 
\end{proposition}

\begin{proof}
	Because $f$ is uniformly continuous, and $\mesh(K_{n}) \to 0$, for every $m \in \NN$, we can find $n \geq m$ such that $\sqsp_{f,m,n}$ is co-surjective. Then we use the above remark.
\end{proof}

For $m \leq n$ and $\sqsp \in \GGRb^{K_m}_{K_n}$, we define
\[ [\sqsp]=\{f \in \Homeo(X):  \sqsp_{f,m,n}=\sqsp \}. \]
Obviously, $f \in [ \sqsp_{f,m,n}]$. The family $\KKb_X=\{K_n\}$ is a sub-paracategory of $\GGRb$ with morphisms that are all co-bijective relations of the form $\sqsp_{f,m,n}$, where $f \in \Homeo(X)$, $m \leq n$.
We may write $\KKb$ for $\KKb_X$ if $X$ is clear from the context, and $\KKb^m_n$ for $(\KKb_X)^{K_m}_{K_n}$. Let $\mathtt{I} \subseteq \KKb_X$ be the wide subcategory with morphisms that are all inclusion morphisms $$\succeq^m_n=\sqsp_{\Id_X,m,n},$$  $m \leq n$. Note that because $K_m$ consolidates $K_n$, if $m \leq n$, we have $$\succeq^k_m \circ \succeq^m_n = \succeq^k_n,$$ for $k \leq m \leq n$, i.e., $\It$ is indeed a category.  When the indices $m$, $n$ are clear from the context, we may write $\succeq$ for $\succeq^m_n$. 

%

%
%
%

For $\sqsp_{f,m,n} \in \mor(\KKb_X)$, let $\la_{f,m,n} \subseteq K_m \times K_m$ be defined by $$\la_{f,m,n}=\sqsp_{f,m,n} \circ \preceq^m_n.$$ 
Thus, objects of $\KKub_X$ are digraphs of the form $\la_{f,m,n}$, where $f \in \Homeo(X)$, $m \leq n$. 

We clearly have that $\la_{f,m,n} \subseteq \la_{f,m,n+1}$, and there exists $n_0 \geq m$ such that $$\la_{f,m,n_0}=\la_{f,m,n_0+1},$$  so we define $$\la_{f,m}=\la_{f,m,n_0}.$$ 
\begin{remark}
By compactness of $X$, $\la_{f,n}$ can be also defined by
$$ p \la_{f,n} q \iff \exists x \in q (p \ni f(x)),$$
where  $p, q \in K_n$.
\end{remark}

For $n \in \NN$, and  a $\KKb$-digraph $\la \subseteq K_n \times K_n$, we define $n_\la=n$, and
\[ [\la]=\{ f \in \Homeo(X): \forall x \in X \, \exists p,q \in K_\la (f(x) \in p \mbox{ and } p \la q \mbox{ and } q \ni x )  \}. \]
Obviously, $f \in [\la_{f,m,n}]$, provided that $\sqsp_{f,m,n} \in \mor(\KKb_X)$.

Let $G \leq \Homeo(X)$ be a subgroup of $\Homeo(X)$. The paracategory $\KKb_G \subseteq \KKb_X$ is defined as the wide sub-paracategory with morphisms that are all the morphisms $\sqsp \in \mor(\KKb_X)$ satisfying $[\sqsp] \cap G \neq \emptyset$. Clearly, $\It \subseteq \KKb_G$. We write $[\sqsp]_G$, $[\la]_G$ for $[\sqsp] \cap G$, $[\la] \cap G$, respectively. For $\LLb \subseteq \GGRb$, let 
\[ G_\LLb=\{f \in \Homeo(X):  \sqsp_{f,m,n} \in \mor(\GGRb) \Rightarrow \sqsp_{f,m,n} \in \mor(\LLb)\}. \]
%

\begin{proposition}
\label{pr:LaxAbs}
For every $G \leq \Homeo(X)$, the paracategory $\KKb_G$ is  $\It$-lax-absorbing.
\end{proposition}

\begin{proof}
Let $f \in G$, and let $k,m \in \NN$ be such that $\sqsp_{f,k,m} \in \KKb_G$.  Let $n \in \NN$ be such that $\sqsp_{f^{-1},m,n} \in \KKb_G$. Then $\sqsp_{f,k,m} \circ \sqsp_{f^{-1},m,n} \subseteq \succeq^k_n$.
\end{proof}

\begin{proposition}
\label{pr:Subset}
Let $G \leq \Homeo(X)$. If $\twoheadleftarrow \in \KKub_G$ is an $\It$-lift of $\la \in \KKub_G$, then $[\twoheadleftarrow]_G \subseteq [\la]_G$.
\end{proposition}

\begin{proof}
For $f \in [\twoheadleftarrow]$, $x \in X$, we fix $p,q \in K_\twoheadleftarrow$ such that $$f(x) \in p, \quad p \twoheadleftarrow q, \quad q \ni x.$$ Because $\succeq^{n_\la}_{n_\twoheadleftarrow}$ is a morphism from $\twoheadleftarrow$ to $\la$, i.e., it is weakly arrow-preserving, there are $p', q' \in K_\la$ such that $$ p \preceq^{n_\la}_{n_\twoheadleftarrow} p', \quad p' \la q', \quad q' \succeq^{n_\la}_{n_\twoheadleftarrow} q,$$ i.e., $$f(x) \in p', \quad p' \la q', \quad q' \ni x.$$ As $x \in X$ was arbitrary, we get that $f \in [\la]$.
\end{proof}

\begin{proposition}
\label{pr:Basis}
Let $G \leq \Homeo(X)$. For every $\la \in \KKub_G$, $[\la]_G$ is open in $G$. Moreover, for every $f \in G$, and $\epsilon>0$, there exists $n \in \NN$ such that $\diam([\la_{f,n}])<\epsilon$.  In particular,  the collection $$\{[\la]_G: \la \in \KKub_G\}$$ is a basis of $G$. 
\end{proposition}

\begin{proof}
Let $\la \in \KKub_G$. To show that $[\la]_G$ is open, we fix $f \in [\la]_G$, and suppose that there exist $f_i \in G$, $i \in \NN$, such that $f_i$ converge  to $f$ but  $f_i \not \in [\la]_G$, $i \in \NN$. We fix $x_i \in X$ witnessing it, i.e., such that there are no $p,q \in K_\la$ with $$f_i(x_i) \in p, \quad p \la q, \quad q \ni x_i.$$ Without loss of generality, we can assume that $x_i$ converge to some $x \in X$. Then $f_i(x_i)$ converge to $f(x)$. However, as $f \in [\la]_G$, there are $p, q \in K_\la$  such that $$f(x) \in p, \quad p \la q, \quad  q \ni x,$$ and the same would also hold for almost all $x_i$; a contradiction.

Now, for a fixed $f \in G$, and $\epsilon>0$, let $\delta>0$ be such that $d(x,x')<\delta$ implies $d(f(x),f(x'))<\epsilon/3$, and let $n \in \NN$ be such that $\mesh(K_n)<\delta$. Then, for every $q \in K_n$,
$$\diam(\bigcup \{p \in K_n: p \la_{f,n} q)\})<\epsilon,$$
so $d(g,g')<\epsilon$ for any $g,g' \in [\la_{f,n}]$. In other words, $\diam([\la_{f,n}])<\epsilon$.
\end{proof}

\begin{remark}
	\label{re:ExtendMor}
	Let $m, n \in \NN$, and let $\sqsp, \sqni \in \KKb^m_n$ be such that $\sqsp \subseteq \sqni$. Let $\la=\sqsp \circ \preceq^m_n$, $\twoheadleftarrow=\sqni \circ \preceq^m_n$. It is straightforward to observe that	 
	\begin{enumerate}
			\item $[\la] \subseteq [\twoheadleftarrow]$,
		\item if $f \in [\sqsp]$, $g \in [\sqni]$, then $d(f,g)<\mesh(K_m)$,	
		\item if $\mesh(K_m)<\epsilon$, and $\diam([\la])<\epsilon$, then $\diam([\twoheadleftarrow])<3\epsilon$.   
	\end{enumerate}
\end{remark}

\begin{proposition}
	\label{pr:DownwardClosed}
	
	Let $G \leq \Homeo(X)$, and let $\LLb=\KKb_G$. For any $\la, \twoheadleftarrow \in \LLub$, $\sqsp \in \LLub^\la_{\tla}$, $f \in [\twoheadleftarrow]_G$, and $g\in [\sqsp]_G$, we have that $$gfg^{-1} \in [\la]_G.$$ 
\end{proposition}

\begin{proof}
	Fix $\la, \twoheadleftarrow \in \LLub$, $\sqsp \in \LLub^\la_{\tla}$, $f \in [\twoheadleftarrow]_G$, $g\in [\sqsp]_G$.
	We fix $x \in X$, and $p,q \in K_\twoheadleftarrow$ such that $$f(g^{-1}(x)) \in p, \quad  p \twoheadleftarrow q, \quad q \ni g^{-1}(x).$$ As $\sqsp$ is weakly arrow-preserving, there are $p', q' \in K_\la$ such that $$p \sqsb p', \quad p' \la q', \quad q' \sqsp q.$$ But then $$g(f(g^{-1}(x))) \in p', \quad  p' \la q', \quad q' \ni g(g^{-1}(x))=x.$$ Because $x \in X$ was arbitrary, we get that $gfg^{-1} \in [\la]$. 
\end{proof}

\begin{proposition}
	\label{pr:InMorphism}
	Let $G \leq \Homeo(X)$, and let $\LLb=\KKb_G$. Let $\la \in \LLub$, and $f,g \in G$. If $gfg^{-1} \in [\la]_G$, then there exists $n_0 \in \NN$ such that $$\sqsp_{g,n_\la, n} \in \LLub^\la_{\la_{f,n}},$$ for every $n \geq n_0$.
\end{proposition}

\begin{proof}
	Fix $x \in X$. We say that $n \in \NN$ works for $x$ if there are $p',q' \in K_\la$ such that $$gfg^{-1}(x) \in p', \quad  p' \la q', \quad q' \ni x,$$ and, for $\twoheadleftarrow=\la_{f,n}$, $\sqsp=\sqsp_{g,n_\la,n}$, the following holds. For any $p,q \in K_{\twoheadleftarrow}$ such that $$fg^{-1}(x) \in \cl(p),  \quad p \twoheadleftarrow q, \quad  \cl(q) \ni g^{-1}(x)$$ we have $$p' \sqsp p, \quad q' \sqsp q.$$ Note that if $n$ works for $x$, and $n' \geq n$, then $n'$ also works for $x$. Because $\mesh(K_n) \to 0$, and $g$ is continuous, there exists $n \in \NN$ that works for $x$. Because $gf$ is continuous, if $n$ works for $x$, there is an open $U_x \subseteq X$ such that $x \in U_x$, and $n$ works for all $x' \in U_x$. By compactness of $X$, there exists $n_0 \in \NN$ that works for all $x \in X$. We leave it to the reader to verify that $n_0$ is as required.
\end{proof}

\begin{corollary}
\label{co:LiftBasis}
Let $G \leq \Homeo(X)$, and $\la \in \KKub_G$. Then the collection $$\{[\twoheadleftarrow]_G: \twoheadleftarrow \in \KKub_G \mbox{ is an } \It \mbox{-lift of} \la \}$$ is a basis of $[\la]_G$.
\end{corollary}

\begin{proof}
For a fixed $\la \in \KKub_G$, we apply Proposition \ref{pr:InMorphism} to a fixed $f \in [\la]$, and $g=\Id_X$. Then we use Propositions \ref{pr:Subset} and \ref{pr:Basis}.
\end{proof}

\begin{corollary}
	\label{co:DensePiBasis}
	Let $G \leq \Homeo(X)$, and let $\LLb \subseteq \KKb_G$ be such that $G_\LLb$ is a dense subset of $G$. Then  the collection $$\{[\la]_G: \la \in \LLub\}$$ is a $\pi$-basis of $G$, and, for every $\la \in \LLub$, the collection $$\{[\twoheadleftarrow]_G: \twoheadleftarrow \in \LLub \mbox{ is an } \It \mbox{-lift of} \la \}$$ is a $\pi$-basis of $[\la]_G$.
\end{corollary}

\begin{corollary}
	\label{co:UpwardClosed}
	Let $G \leq \Homeo(X)$, and let $\LLb \subseteq \KKb_G$ be weakly $\It$-absorbing and such that $G_\LLb$ is a dense subset of $G$. Let $\la \in \LLub$. Then there is $n \geq n_\la$ such that for every $\sqsp \in \LLb^n_\bullet$ there are $\sqsp' \in \mor(\LLb)$, and $\twoheadleftarrow \in \LLub$ such that  $$\sqni=\succeq^{n_\la}_n \circ \sqsp \circ \sqsp' \in \mor(\It),$$ and $\twoheadleftarrow$ is a $\sqni$-lift of $\la$. In particular, $\tla$ is an $\It$-lift of $\la$.
\end{corollary}

\begin{proof}
 Using the assumption that $\LLb$ is weakly $\It$-absorbing, we fix $n \geq n_\la$ such that for every $\sqsp \in \LLb^n_\bullet$ there is $\sqsp'_0 \in \mor(\LLb)$ such that $$\succeq^{n_\la}_n \circ \sqsp \circ \sqsp'_0 \in \mor(\It).$$ Then we apply Proposition \ref{pr:InMorphism} to a fixed $f \in [\la]$, and $g=\Id_X$, to find $n',n'' \in \NN$ such that, for $\sqsp'=\sqsp'_0 \circ \succeq^{n'}_{n''}$, we have  that $\twoheadleftarrow=\la_{f,n''}$ is a $\sqni$-lift of $\la$.
\end{proof}

\section{Conjugacy classes in homeomorphism groups of compact metric spaces}
Recall that a subset $A \subseteq X$ of a Polish space $X$ is \emph{meager} if there are closed nowhere dense $F_n \subseteq X$ such that $A \subseteq \bigcup_n F_n$. It is \emph{comeager} if $X \setminus A$ is meager. It is well known that the intersection of a countable family of comeager sets is dense. Also,  the Baire category theorem holds for Polish spaces, i.e., if $A$ has the Baire property, then  $A$ is non-meager iff there exists a non-empty open $U \subseteq X$ such that $A \cap U$ is comeager as a subset of the Polish space $U$. See \cite{Gao} for more information on this topic.

We start with a citerion for the existence of a non-meager orbit under a continuous action of a Polish group on a Polish space. It is a variant of an analogous criterion for the existence of a co-meager orbit due to C. Rosendal. Note that the conjugation action of a Polish group on itself is continuous.

\begin{proposition}
	\label{pr:Criterion}
	Let $Y$ be a Polish space and $G$ be a Polish group acting continuously on $Y$. Then the following are equivalent:
	\begin{enumerate}
		\item Suppose that the following condition holds:
		
		\medskip
		
		\noindent for any open $V \ni 1_G$, and non-empty open $U \subseteq Y$ there exists a non-empty open $U' \subseteq U$ such that for any non-empty open $W_1, W_2 \subseteq U'$, we have
		\[
		V \cdot W_1 \cap W_2 \neq \emptyset.
		\]
		\medskip
		Then there exists a non-meagre $G$-orbit.
		\item Suppose that there exists a comeagre $G$-orbit. Then the above condition holds.
	\end{enumerate}
\end{proposition}

\begin{proof}
	
	We prove (1).
	Suppose that all orbits are meagre. Let $x \in Y$ be arbitrary.
	As $G \cdot x$ is meagre, there exist closed nowhere dense sets
	$\{F_n : n \in \mathbb{N}\}$ such that
	\[
	G \cdot x \subseteq \bigcup_{n} F_n.
	\]
	Let
	\[
	B_n = \{ g \in G : g \cdot x \in F_n \}.
	\]
	Then each $B_n$ is closed, and by the Baire category theorem there exists $n$
	such that $B_n$ has non-empty interior.
	Let $V \ni 1_G$ be open and $g \in G$ such that $gV \subseteq B_n$.
	Then
	\[
	V \cdot x \subseteq g^{-1} \cdot F_n,
	\]
	so $V \cdot x$ is nowhere dense.
	We conclude that for every point $x \in Y$ there exists $V \ni 1_G$ such that
	$V \cdot x$ is nowhere dense.
	
	As there are only countably many possible $V$ (we may restrict to a basis at $1_G$),
	there exists $V \ni 1_G$ and a non-empty open set $U \subseteq X$ such that the set
	\[
	\{ x \in U : V \cdot x \text{ is somewhere dense} \}
	\]
	is meagre (Here we use the fact that $\{ x \in Y : V \cdot x \text{ is nowhere dense} \}$ is a Borel set and hence has the property of Baire, and that a non-meagre set with the Baire property must be comeagre in an open set). Let now $U' \subseteq U$ be as given by the condition in~(2).
	The property of $U'$ implies that the set
	\[
	\{ x \in U' : V \cdot x \text{ is dense in } U' \}
	=
	\{ x \in U' : \forall W_1 \subseteq U',\ V \cdot x \cap W_1 \neq \varnothing \}
	\]
	is a dense $G_\delta$ subset of $U'$, which is a contradiction
	(the quantifier $\forall W_1 \subseteq U'$ can be taken to range over a countable
	basis of $U'$).
	
	Now we prove (2).
	As orbits are Borel, and so they have the Baire property, there exists $x_0 \in Y$, and a non-empty open $U \subseteq X$ such that $G \cdot x_0 \cap U$ is co-meagre in $U$. Fix open $V \ni 1_G$, and 
	let $V' \ni 1_G$ be open such that $V' V'^{-1} \subseteq V$.
	Let $x \in U \cap G \cdot x_0$.
	By Effros's theorem \cite[Theorem ???]{Gao}, there exists an open set $U_0 \subseteq Y$ such that
	\[
	V' \cdot x = U_0 \cap G \cdot x,
	\]
	and in particular $U_0 \subseteq V' \cdot x$.
	Finally, let $U' = U_0 \cap U$ and let $W_1, W_2 \subseteq U'$ be non-empty open sets.
	Then there exist $v_1, v_2 \in V'$ such that
	\[
	v_1 \cdot x \in W_1
	\quad\text{and}\quad
	v_2 \cdot x = v_2 v_1^{-1} (v_1 \cdot x) \in V \cdot W_1 \cap W_2.
	\]
	
\end{proof}

\begin{theorem}
\label{th:Dense} 
Let $G \leq \Homeo(X)$ be closed, and let $\LLb \subseteq \KKb_G$ be such that $G_\LLb$ is a dense subset of $G$. Then  $\LLb$ is directed iff $G$ has a dense conjugacy class. 
\end{theorem}

\begin{proof}
Fix $A_0, A_1 \in \LLub$. Since $G$ has a dense conjugacy class, there are $f \in G$, and  $g_0,g_1 \in G_\LLb$ such that $$g_0fg_0^{-1} \in [A_0]_G, \quad g_1fg_1^{-1} \in [A_1]_G.$$ By Proposition \ref{pr:InMorphism}, there is $n \in \NN$ such that  $$\sqsp_{g_0,n_{A_0}, n} \in (\KKub_G)^{A_0}_{\la_{f,n}}, \quad \sqsp_{g_1,n_{A_1}, n} \in (\KKub_G)^{A_1}_{\la_{f,n}}.$$ But $G_\LLb$ is a dense subset of $G$, so there is also $f \in G_\LLb$ so that the above holds.

Now suppose that $\LLub$ is directed, and fix $A_0, A_1 \in \LLub$. We fix $B \in \LLub$, and morphisms $\sqsp_0 \in \LLub^{A_0}_B$, $\sqsp_1 \in \LLub^{A_1}_B$. By Proposition \ref{pr:DownwardClosed},  there are $f_0 \in [A_1]$,  $g_0 \in G$ such that $g_0 f_0 g^{-1}_0 \in [A_0]_G$. Moreover, by continuity of group operations in $G$, there is an open $U\subseteq [A_1]_G$ such that $g_0fg^{-1}_0 \in [A_0]$ for every $f \in U$.  As $A_1 \in \LLub$ was arbitrary, it follows that the set of all $f \in G$ such that there is $g \in G$ with $g f g^{-1} \in [A_0]_G$ is open and dense in $G$. As $A_0 \in \LLub$ was arbitrary, the set of $f \in G$ with dense conjugacy class is comeager; in particular, it is not empty. 
\end{proof}

\begin{theorem}
\label{th:AmalNonMeager}
Let $G \leq \Homeo(X)$ be closed, and let $\LLb \subseteq \KKb_G$ be such that $G_\LLb$ is a dense subset of $G$. If $\LLb$ is weakly $\It$-absorbing, and $\LLub$ has weak m-amalgamation, then $G$ has a non-meager conjugacy class.  
\end{theorem}

\begin{proof}
We verify condition (1) of Proposition~\ref{pr:Criterion}. Fix $V \ni 1_G$, and let $W \ni 1_G$ be open, symmetric, and such that $W^2 \subseteq V$. Fix $A \in \LLub$. Without loss of generality, we can assume that $[I]_G \subseteq W$, where $I=\Id_{K_A}$ is the identity digraph on $K_A$. Fix $B \in \LLub$, and $\nil \in \LLub^A_B$ that form a weak amalgamation basis for $A$. By the fact that $\LLb$ is weakly $\mathtt{I}$-absorbing, and Corollary \ref{co:UpwardClosed}, we can assume that $\nil \in \mor(\It)$. 

By Corollary \ref{co:DensePiBasis}, the collection $\{[C]_G: C \in \LLub \mbox{ is an } \It \mbox{-lift of } B  \}$ is a $\pi$-basis of $[B]_G$. Fix $\It$-lifts $C_0, C_1$ of $B$, and fix an m-amalgam $D \in \LLub$ with morphisms $\sqsp'_0 \in \LLub^{C_0}_D$, $\sqsp'_1 \in \LLub^{C_1}_D$. 
%
%
By Corollary \ref{co:UpwardClosed}, there are $E \in \LLub$ and $\sqni \in \LLub^D_E$ such that
%
$$ \succeq^{n_A}_{n_{C_1}} \circ \sqsp'_1 \circ \sqni = \succeq_{n_E}^{n_A}.$$
Hence, the above equality implies that if $g \in [\sqsp'_1 \circ \sqni]_G$, then $g \in [I]_G$. Moreover, by m-amalgamation, for every $q \in K_D$ there is $p \in K_A$ such that $$p \succeq \circ \sqsp'_0 q, \quad  p \succeq \circ \sqsp'_1 q.$$ Consequently, we also have that for every $q \in K_E$ there is $p \in K_A$ such that $$p \succeq \circ \sqsp'_0 \circ \sqni q, \quad p \succeq \circ \sqsp'_1 \circ \sqni q,$$ i.e., $p \succeq q$. Hence, if $g \in [\sqsp'_0 \circ \sqni]_G$, then $g \in [I]_G$. 

In particular, by Proposition \ref{pr:DownwardClosed}, there are $f \in [C_0]_G$, and $g \in W^2 \subseteq V$ such that $gfg^{-1} \in [C_1]_G$, i.e., $g[C_0]_Gg^{-1} \cap [C_1]_G \neq \emptyset$. Moreover, $[C_0]_G$, $[C_1]_G$ as above form a $\pi$-basis of $[B]$.  Since the sets $[A]_G$, $A\in\LLub$, form a $\pi$-basis of $G$,
this verifies condition (1) of Proposition~\ref{pr:Criterion}.
\end{proof}

\begin{theorem}
	\label{th:ComeagerDir&Amal}
Let $G \leq \Homeo(X)$ be closed, and let $\LLb \subseteq \KKb_G$ be weakly $\It$-absorbing and such that $G_\LLb$ is a dense subset of $G$. Then $\LLub$ is directed and has weak m-amalgamation iff $G$ has a comeager conjugacy class.
\end{theorem}

\begin{proof}
By Theorems \ref{th:Dense} and \ref{th:AmalNonMeager}, we only need to prove that the existence of a comeager conjugacy class implies weak m-amalgamation. 
Fix $A \in \LLub$, and  set $\epsilon=\ell(K_A)$. Fix $f \in [A]_G$ with comeager conjugacy class, and an open $V \ni 1_G$ such that $$\{gfg^{-1}: g \in V \} \subseteq [A]_G, \quad \diam(V)<\epsilon/4.$$ By Proposition \ref{pr:Criterion}(2) and Corollary \ref{co:DensePiBasis}, there is an $\It$-lift $B$ of $A$ such that $\{gfg^{-1}: g \in V \}$ is dense in $[B]_G$. 
 
 Fix $C_0,C_1 \in \LLub$, and $\sqsp_0 \in \LLub^B_{C_0}$, $\sqsp_1 \in \LLub^B_{C_1}$. We can assume that $\mesh(C_0), \mesh(C_1)< \epsilon/4$, and, by Corollary  \ref{co:UpwardClosed}, that $\sqsp_0, \sqsp_1  \in \mor(\It) $, i.e., $[C_0]_G, [C_1]_G \subseteq [B]_G$. Since $\{gfg^{-1}:g \in V\}$ is dense in $[B]_G$, there are $g_0, g_1 \in V$ such that $$g_0fg^{-1}_0 \in [C_0]_G, \quad g_1fg^{-1}_1 \in [C_1]_G.$$ 
 
Since $[C_0]_G$ and $[C_1]_G$ are open, we can further choose $f_0\in G_{\LLb}$ sufficiently close to $f$ so that
 \[
 g_0f_0g_0^{-1}\in[C_0]_G,\quad
 g_1f_0g_1^{-1}\in[C_1]_G.
 \]
 By Proposition \ref{pr:InMorphism}, there is $n \in \NN$, such that, for $D=\la_{f_0,m,n}$, $\sqsp'_0=\sqsp_{g_0,m,n} \in \LLub^{C_0}_D$, $\sqsp'_1=\sqsp_{g_1,m,n} \in \LLub^{C_1}_D$. 
Consequently, for every $r  \in K_D$, and $q_0 \in K_{C_0}$, $q_1 \in K_{C_1}$ such that 
$q_0 \sqsp'_0 r$, $q_1 \sqsp'_1 r$ we have
 $$\diam(q_0 \cup q_1)<\epsilon,$$ so for every $r \in K_D$, there is $p \in K_A$ such that $$p \succeq \circ \sqsp'_0 r, \quad p \succeq \circ \sqsp'_1 r,$$ i.e., $$p \succeq \circ \sqsp_0 \circ \sqsp'_0 r, \quad p \succeq \circ \sqsp_1 \circ \sqsp'_1 r.$$ In other words, $D$ is the required m-amalgam.
\end{proof}

\section{Spectra of posets, and \fra \ limits}
In this section, instead of selecting a compact space with an appropriate sequence of its open covers, and then defining corresponding (para)categories of graphs, we will proceed in the opposite direction. We will select a category of graphs, an appropriate sequence in it, and construct a compact space so that this sequence will be, essentially, a sequence of its covers. This approach -- a topological variant of \fra \ theory -- has been described in detail in \cite{BaBiVi2}. We also refer the reader to Section 3 in \cite{Ma} for a simplified treatment that is sufficient in the present context. Here, we recall only basic notions and facts required to carry out the construction.

\subsection{Spectra od posets}
Let $\KKb \subseteq \GGRb$ be a category of graphs, and let $(K_n, \succeq_n)$ be a sequence in $\KKb$. We assume that $K_n$'s are pairwise disjoint. Let $\KK=\bigcup_n K_n$, and let $\succeq$ be the reflexive and transitive closure of $\bigcup_n \succeq_n$. Then $(\KK,\succeq)$ is clearly a poset, called the poset \emph{induced} by $(K_n, \succeq_n)$, and sometimes referred to also as $(K_n, \succeq_n)$; $K_n$'s are called \emph{levels} of $\KK$. Note that because all $\succeq_n$ are surjective, $(\KK,\succeq)$ has no minimal elements. The function $r:\KK \to \NN$, where $r(p)$ is the unique $n \in \NN$ such that $p \in K_n$, is a rank function assigning $0$ to maximal elements. Obviously, $r^{-1}(n)=K_n$; in particular $r^{-1}(n)$ is finite, for $n \in \NN$. Posets $P$ admitting such a rank $r:P \to \NN$ with finite preimages $r^{-1}(n)$, $n \in \NN$, are called \emph{$\omega$-posets} in \cite{BaBiVi}, and posets induced by sequences $(K_n, \succeq_n)$ are called \emph{$\omega$-chains} in \cite{Ma}. 

A subset $C \subseteq \KK$ is called a \emph{cap} if it is  $\preceq$-refined by some level $K_n$ (caps are defined somewhat differently in \cite{BaBiVi} but for $\omega$-chains these two definitions are equivalent). We call $S \subseteq \KK$ a \emph{selector} if it has non-empty intersection with every cap (equivalently: $\KK \setminus S$ is not a cap).  The \emph{spectrum} $\Sp \KK$ of $\KK$ is the family of all minimal selectors in $\KK$. The reader is referred to the introductory part of Section 3 in \cite{Ma} for some motivation for the notions of cap, selector, and spectrum.

 For $p \in \KK$, let $$p^\in=\{S \in \Sp \KK: p \in S \},$$ and, for $A \subseteq \KK$, let $$A^\in=\{p^\in: p \in A \}.$$ The spectrum $\Sp \KK$ is equipped with the topology $\tau_\KK$ generated by the sub-basis $\KK^\in$ (which in fact is a basis, see \cite[Corollary 2.14]{BaBiVi}).  By \cite[Proposition 2.8]{BaBiVi}, the topology $\tau_\KK$ is $T_1$ compact. Moreover, if $\KK$ is a regular $\omega$-poset (see \cite{BaBiVi} or \cite{Ma} for the definition and basic properties), then $\Sp \KK$  is regular, i.e., metrizable. By \cite[Propositions 1.7]{BaBiVi} and \cite[Propositions 2.8]{BaBiVi} (see also \cite[Proposition 3.12]{Ma}), $(K^\in_n)$ is a sequence of minimal open covers of $\Sp \KK$ such that
\begin{enumerate}
	\item each $K^\in_n$ consolidates $K^\in_{n+1}$,
	\item $\mesh_d(K^\in_{n}) \to 0$ for every compatible metric $d$ on $\Sp \KK$.
\end{enumerate}
Additionally, $(K^\in_n,\supseteq \uhr K^\in_{n} \times K^\in_{n+1})$ is isomorphic with $(K_n, \succeq_n)$ (see \cite[Corollary 3.17]{Ma}), so we can identify these two sequences, and the induced posets. 

\subsection{\fra \ categories and their limits} Now suppose that $\KKb$ is a \fra \ category, and $(K_n, \succeq_n)$ is a \fra \ sequence in $\KKb$. Let $\KK$ be the poset induced by $(K_n, \succeq_n)$.  
By \cite[Proposition 4.32]{BaBiVi2}, any two \fra \ sequences in $\KKb$ have homeomorphic spectra, so we will call the spectrum $\Sp \KK$ of $\KK$ the \emph{\fra \ limit} of $\KKb$, and denote it by $\Sp \KKb$. Clearly, if $(K_n, \succeq_n)$ is a \fra \ sequence in $\KKb$, then $(K^\in_n,\supseteq \uhr K^\in_{n} \times K^\in_{n+1})$ is  also a \fra \ sequence in $\KKb$. See \cite[Theorem 4.34]{BaBiVi2} for conditions on morphisms in $\KKb$ that warrant metrizability of its \fra \ limit. The next proposition provides basic facts about $\KKb$, $\KKub$, and their relationships with $\KKb_{\Sp \KKb}$, $\KKub_{\Sp \KKb}$.

\begin{proposition}
\label{pr:Robust}
Let $\KKb \subseteq \GGRb$ be a \fra \ category with a subcategory $\It \subseteq \KKb$ induced by a \fra \ sequence $(K_n, \succeq_n)$ in $\KKb$, and let $X=\Sp \KKb$. Let the paracategory $\KKb_{X}$ be given by the sequence $(K^\in_n)$ of minimal open covers of $X$. 
\begin{enumerate}
	\item for any $m \leq n$, and $\sqsp \in \KKb^m_n$, we have $[\sqsp]_{G_\KKb} \neq \emptyset$; in particular, $\sqsp \in \mor(\KKb_X)$, and $\la=\sqsp \circ \preceq_n^m \in  \KKub_X$,
	\item $\cl(G_\KKb)$ is a subgroup of $\Homeo(X)$, 
	\item for any $\tla \in \KKub$, $n \in \NN$, and $\sqni \in \KKb^n_{n_\tla}$,  there is $\la \in \KKub$ such that $K_\la=K_n$, and $\sqni \in \KKub^\la_{\tla}$, i.e., $\tla$ is a $\sqni$-lift of $\la$,
	\item for any $\la \in \KKub$, $n \in \NN$, and $\sqni \in \KKb^{n_\la}_n$, there is $\tla \in \KKub$ such that $K_\tla=K_n$, and $\sqni \in \KKub^\la_{\tla}$, i.e., $\tla$ is a $\sqni$-lift of $\la$,
	\item suppose $G \leq \Homeo(X)$ is such that  for every ${\sqsp} \in (\KKb_X)_G$ there is ${\sqni} \in \KKb$ such that ${\sqsp} \subseteq {\sqni}$. Then $G_\KKb$ is dense in $G$; in particular, if $G=\Homeo(X)$, then $\cl(G_\KKb)=\Homeo(X)$.
\end{enumerate}	
\end{proposition}

\begin{proof}
(1): Let $\KK$ be the poset  induced by $(K_n, \succeq_n)$. Fix $\sqsp \in \KKb^{m_0}_{n_0}$.  Using absorption, we select $m_k, n_k \in \NN$, for $k>0$, and $\sqni_k \in \KKb^{m_k}_{n_k}$, $\sqni'_k \in \KKb^{n_k}_{m_{k+1}}$ such that $\sqni_0=\sqsp$, and $$\sqni_k \circ \sqni'_{k}=\succeq, \quad \sqni'_k \circ \sqni_{k+1}=\succeq.$$
Let $\sqni=(\bigcup_k \sqni_k)^\preceq$, $\sqni'=(\bigcup_k \sqni'_k)^\preceq$. We want to use \cite[Proposition 3.7]{BaBiVi}. It is easy to check that $\sqni$, $\sqni'$ are refiners (i.e., $C^{\sqni}$ and $C^{\sqni'}$ are caps, whenever $C \subseteq \KK$ is a cap) such that $$\sqni' \circ \sqni=\succeq.$$ By \cite[Proposition 3.7]{BaBiVi}, the mappings $f$, $g$ defined by $S \mapsto S^{\sqin}$, $S \mapsto S^{\sqin'}$ are homeomorphisms of $X$ such that $g=f^{-1}$, $\sqsp_{f,m_0,n_0}=\sqni_0=\sqsp$, and $f,g \in G_\KKb$. 
 
(2): Fix $\epsilon>0$, $k \leq m \leq n \in \NN$, $\sqsp \in \KKb^k_m$ such that $\mesh(K^\in_k)<\epsilon$, $\sqsp' \in \KKb^m_n$, and $f \in [\sqsp]$, $f' \in [\sqsp']$. Clearly, $\sqsp \circ \sqsp' \subseteq \sqsp_{f \circ f',k,n}$. By (1), there is $g \in [\sqsp \circ \sqsp']_{G_\KKb}$, and, by Remark \ref{re:ExtendMor}(2), $d(f \circ f',g)<\mesh(K^\in_k)<\epsilon$. Thus $\cl(G_\KKb)$ is closed under products. Moreover, by (1), every element of $G_\KKb$ can be approximated arbitrarily closely
by $g\in G_\KKb$ such that $g^{-1}\in G_\KKb$. Hence $\cl(G_\KKb)$ is also closed under inverses, and therefore is a subgroup.

(3): We define $$\la=\sqni \circ \sqsp_1 \circ \sqsb_0 \circ \sqin,$$ where $\sqsp_0,  \sqsp_1$ are such that $\tla=\sqsp_1 \circ \sqsb_0$, and use Remark \ref{re:FraisseDigraphs}.

(4): Fix $\la=\sqsp \circ \preceq$, where $\sqsp \in \KKb_l^{n_\la}$ for some $l \in \NN$. Suppose first that $n:=n_\la=l$, and $\sqni \in \It$. Using absorption, we fix $n' \in \NN$ and $\sqsp' \in \KKb^l_{n'}$ such that $$\sqsp \circ \sqsp'= \succeq^{n_\la}_{n'}.$$ Then $$\twoheadleftarrow=\succeq^{l}_{n'} \circ \sqsb'$$ is easily verified to be as required. If $n>l$, we can actually assume that $n=l$ by replacing $\sqsp$ with $\sqsp \circ \succeq^l_{n}$; if $n<l$, in the last step we apply (3) to $\tla$ constructed as above and $\succeq^{n}_l$. Finally, if $\sqni \not \in \It$, using absorption, we apply the above construction to $\sqni \circ \sqni' \in \It$, and then apply (3) to $\tla$ and $\sqni'$.  

(5): We note that, by (1), $[\la]_{G_\KKb} \neq \emptyset$, for every $\la \in \KKub$ such that $K_\la=K_n$ for some $n \in \NN$.  Thus, by our assumption and Remark \ref{re:ExtendMor}(2), $\{[\la]: \la \in \KKub \}$ is a $\pi$-base of $G$, i.e., $G_\KKb$ is dense in $G$.
\end{proof}

In the context of \fra \ limits, Theorem \ref{th:Dense} and Theorem \ref{th:ComeagerDir&Amal} can be formulated as follows:

\begin{theorem}
\label{th:ConjugacyForFraisse}
	Let $\KKb \subseteq \GGRb$ be a  \fra \ category, and let $G=\cl(G_\KKb) \leq \Homeo(\Sp \KKb)$.
	\begin{enumerate}
	\item  $G$ has a dense conjugacy class iff $\KKub$ is directed,
	\item $G$ has a comeager conjugacy class iff $\KKub$ is directed, and has weak m-amalgamation.
	\end{enumerate}
\end{theorem}

\section{Applications}
\subsection{The arc}
In this section, we reprove Hjorth's result that the group of increasing homeomorphisms of the unit interval has a comeager conjugacy class. The original proof (see \cite[Theorem 4.6]{Hj}) explicitely characterizes the comeager class. Here, we show that an appropriate category of digraphs has weak m-amalgamation.

For a graph $(K,\sqcap) \in \GGRb$, and $x \in K$, the \emph{degree} of $x$ is $$|\{x' \in K: x' \sqcap x, \, x' \neq x\}|.$$ A \emph{path} is a finite, acyclic and connected graph all of whose vertices have degree at most $2$.  Every path has exactly two vertices of degree $1$, called \emph{ends}, and exactly two linear orderings that are compatible with the graph relation (i.e., the immediate successor of every element is its neighbor). We associate with every path such an ordering $\leq$, i.e., we regard paths as linearly ordered sets, and we use the notation $[x,y]$, $(x,y)$, $x+1$, etc., without further comments. Alternatively, we can think of paths as intial segments $\{0,\ldots, n\}$ of the natural numbers $(\NN, \leq)$ with the standard ordering $\leq$. Somewhat abusing notation, we always denote the smallest element of a path $K$ by $\rt$, and the largest element by $\max K$.

For $K, L \in \GGRb$, we say that a relation $R \subseteq K \times L$, is \emph{monotone} if the set $C^R$ is connected (as an induced subgraph of $L$), whenever $C \subseteq K$ is connected. If $K$ and $L$ are paths, we say that $R$ is \emph{$\leq$-monotone} if it is $\leq $-preserving or $\geq $-preserving.  

\begin{remark}
Let $K=\{0,1\}$, $L=\{0,1,2, 3\}$, and $\sqsp= K \times L \setminus \{(0,1), (1,0)\}$. Then $\sqsp \in \GGRb^K_L$ is monotone but not $\leq$-monotone. 
\end{remark}

\begin{remark}
\label{re:ConnectedBasic}
Let $K$, $L$ be paths, and let $\sqsp \in \GGRb^K_L$. It is straightforward to observe that, because $\sqsp$ is edge-preserving,
\begin{enumerate}
	\item $C \subseteq L$ is connected iff it is a path, i.e., an interval $[x,x']$,
	\item if $C \subseteq L$ is connected, then $C^\sqsb$ is connected,
	\item for every $l \in L$ there are unique neighbors $\alpha^\sqsb_l, \beta^\sqsb_l \in K$ (possibly $\alpha^\sqsb_l= \beta^\sqsb_l$) such that $$\{l\}^\sqsb=\{\alpha^\sqsb_l, \beta^\sqsb_l\}, \quad \alpha^\sqsb_l \leq \beta^\sqsb_l.$$ We will write $\alpha_l$, $\beta_l$ if $\sqsp$ is clear from the context.
\end{enumerate}
\end{remark}

\begin{proposition}
	\label{pr:CharCMon}
	Let $K, L \in \GGRb$ be paths, and let $\sqsp \in \GGRb^K_L$ be such that $\pi_{\sqsp}(\rt)=\rt$. Then $\sqsp$ is monotone iff the following conditions hold:
	\begin{enumerate}
		\item $\pi_\sqsp(\max L)=\max K$,
		\item  for any $l,l' \in L$, we have that $l \leq l'$ implies $$\alpha_{l} \leq \alpha_{l'}, \quad \beta_{l} \leq \beta_{l'}.$$ 
	\end{enumerate}
\end{proposition} 

\begin{proof}
First, we show that if $\sqsp$ is not monotone, then (2) does not hold. Let us fix $C=[k,k'] \subseteq K$ such that $D= C^\sqsp$ is not connected. Then there exist $l,l'_0,l'_1 \in L$ such that $l'_0<l<l'_1$, $l'_0,l'_1 \in D$, and $l \not \in D$. It is straightforward to verify that either $l, l'_0$ or $l, l'_1$ violate (2). 


Conversely, suppose that $\sqsp$ is monotone.  First we show (2).  It sufficies to consider the case that $l'=l+1$. Fix $l \in L$, and suppose that $\alpha_{l+1}<\alpha_l$. Clearly, $[\rt,l]^\sqsubset$ is an interval containing $\rt$ (in $K$) and $\alpha_l$, so there is $l'<l$ such that $\alpha_{l+1} \sqsp l'$. But then $l$ would  witness that $\{\alpha_{l+1}\}^\sqsp$ is not connected, which would be a contradiction, so $\alpha_{l} \leq \alpha_{l+1}$.  Suppose that $\beta_{l+1}<\beta_l$. Then we must have that $\beta_l=\alpha_l+1$, and $\alpha_{l+1}=\beta_{l+1}=\alpha_l$. Using that $\alpha_{l} \leq \alpha_{l+1}$, we observe that if $\pi_\sqsp(l')=\beta_l$, then $l'>l+1$. Hence, $l+1$ witnesses that $\{\beta_l\}^\sqsp$ is not connected, a contradiction.

Finally, by (2), if $\pi_\sqsp(l)=\max K$, and $l'>l$, then $\pi_\sqsp(l')=\max K$. Since $\pi_\sqsp$ is surjective, this implies $\pi_\sqsp(\max L)=\max K$, so (1) holds.
\end{proof}

\begin{corollary}
\label{co:CharCMon}
Let $K, L \in \GGRb$ be paths, and let $\sqsp \in \GGRb^K_L$.
\begin{enumerate}
	\item If $\sqsp$ is  a function, then $\sqsp$ is monotone iff it is $\leq$-monotone.
	\item  If $\sqsp$ is monotone, then there exists a $\leq$-monotone function $\sqsp' \in \GGRb^K_L$ such that $\sqsp' \subseteq \sqsp$.
\end{enumerate}
\end{corollary} 

\begin{proof}
To show (1), we note that $\sqsp$ is a function iff $\alpha_l=\beta_l$, $l \in L$. To show (2), we define a function $\sqsp':L \to K$, by putting ${\sqsp'}(l)=\beta_l$, and observe that it is surjective because $\pi_\sqsp \subseteq \sqsp'$, and $\sqsp$ is co-injective, i.e., $\pi_\sqsp$ is surjective.
\end{proof}
  
Let $\IIb \subseteq \GGRb$ be the category of all paths, with monotone morphisms $\sqsp \in \mor(\GGRb)$ such that $\pi_\sqsb(\rt)=\rt$ as morphisms. By \cite[Proposition 5.13]{BaBiVi2} and \cite[Theorem 5.15]{BaBiVi2}, $\IIb$ is a \fra \ category, and its \fra \ limit is homeomorphic with  the arc, i.e., the interval $[0,1]$. Let $\It \subseteq \IIb$ be given be a fixed \fra \ sequence in $\IIb$, and let us denote by $\Homeo_+([0,1])$ the clopen subgroup of $\Homeo([0,1])$ consisting of all increasing homeomorphisms, i.e., $f \in  \Homeo([0,1])$ such that $f(0)=0$. In the sequel, we denote $\IIb$-digraphs by arrows $\la$ or $\twoheadleftarrow$ if such notation makes arguments more suggestive, otherwise we use letters $A$, $B$, etc. Also, we denote morphisms that are functions by letters $p$, $\pi$, etc. rather than symbols $\sqsp$, $\sqni$, etc. Recall that, for $K \subseteq L$, a function $p:L \to K$ is called a \emph{contraction} if $p \uhr K=\Id_K$.

  \begin{proposition}
  	\label{co:G_IDense}
  	Let $K,L,M \in \IIb$, and let $\sqsp \in \GGRb^K_L$, $\sqsp' \in \GGRb^L_M$, $\succeq \in \GGRb^K_M$ be such that
  	\begin{enumerate}
  		\item $\pi_{\sqsp}(\rt)=\rt$, $\pi_{\sqsp'}(\rt)=\rt$,
  		\item $\sqsp \circ \sqsp' \subseteq \succeq$,
  		\item $\succeq$ is monotone.
  	\end{enumerate}
  	Then there exists a monotone $\sqni \in \GGRb^K_L$ such that $\sqsp \subseteq \sqni$. In particular, $\cl(G_\IIb)=\Homeo_+([0,1])$. 
  \end{proposition}
  
  \begin{proof}
  	First, we claim that for any $l,l' \in L$, if $l<l'$, then $\alpha^\sqsb_{l}-1 \leq \alpha^\sqsb_{l'}$. Fix $m,m' \in M$ such that $l  \sqsp' m$, $l'  \sqsp' m'$. Observe that if $m \leq m'$, then the claim directly follows from (2) and (3). Suppose that $m'<m$. But then $[\rt,l'] \subseteq [\rt,m']^{\sqsb'}$. In particular,  $l \in [\rt,m']^{\sqsb'}$, i.e., there is $m_0 \leq m'$ such that $l \sqsp' m_0$. The claim follows from the above observation applied to $m_0 \leq m'$.
  	
  	Now we observe that if $l \in L$ is $\sqsp$-co-injective, and $\alpha_{l+1}= \pi_\sqsp(l)-1$, then there exists $l'>l$ that is $\sqsp$-co-injective, and $\pi_\sqsp(l)=\pi_\sqsp(l')$. 
  	Indeed, let $l'' \in L$ be the largest such that $\alpha_{l''}= \pi_\sqsp(l)-1$. But then, because $\sqsp$ is edge-preserving, $\alpha^\sqsb_{l''+1}=\beta^\sqsb_{l''+1}=\pi_\sqsp(l)$, i.e.,  $l'=l''+1$ is as required.
  	
  	In particular, for every such $l \in L$, we can extend $\sqsp$ to $\sqni' \subseteq K \times L$ by putting $$\alpha^\sqsb_{l+1} \sqni' l, \quad \pi_\sqsp(l) \sqni' l+1,$$ i.e., $\alpha^{\sqni'}_{l}=\alpha^{\sqni'}_{l+1}$, $\beta^{\sqni'}_l=\beta^{\sqni'}_{l+1}$. By the above observation, $\sqni'$ is still co-injective. 
  	
  	If now $\beta^{\sqsb}_{l+1}<\beta^{\sqsb}_l$, then $l$ is not $\sqsp$-co-injective, and we extend $\sqsp$ by putting
  	\[
  	\beta^{\sqsb}_l\sqni' l+1.
  	\]
  	Repeating this procedure eliminates all remaining decreases of
  	$\beta$. Thus we eventually obtain a co-injective
  	$\sqni\supseteq\sqsp$ such that both
  	$(\alpha^{\sqni}_l)_{l\in L}$ and
  	$(\beta^{\sqni}_l)_{l\in L}$ are nondecreasing. Hence, by
  	Proposition~\ref{pr:CharCMon}, $\sqni$ is monotone.
  	
  	The last statement follows from the following observations. Clearly, $G_\IIb$ is a subset of $\Homeo_+([0,1])$.  Fix  $f \in \Homeo_+([0,1])$. First, observe that for any $m,n \in \NN$ such that $\sqsp=\sqsp_{f,m,n} \in \KKb_{[0,1]}$, the fact that $f(0)=0$ implies that $\pi_{\sqsp}(\rt)=\rt$. By Proposition \ref{pr:LaxAbs}, there exist $\sqsp' \in  \KKb_{[0,1]}$, and $\preceq \in \It$ such that $\sqsp \circ \sqsp' \subseteq \succeq$, i.e., (2) holds. By the above observation, (1) also holds. As $\preceq$ is monotone, i.e., (3) holds, and $f$ was arbitrary, we can apply Proposition \ref{pr:Robust}(5).
  \end{proof}



\begin{proposition}
\label{pr:CharI}
A relation $R \subseteq K \times K$, $K \in \IIb$, is an $\IIb$-digraph of the form $R=\sqsp_0 \circ \sqsb_1$, where $\sqsp_0, \sqsp_1 \in \mor(\IIb)$ are functions  iff it surjective, co-surjective, $\leq$-monotone, and reflexive on ends.
\end{proposition}

\begin{proof}
The implication from left to right is immediate. To see the converse, fix $R \subseteq K \times K$, $K \in \IIb$, that is surjective, co-surjective, $\leq$-monotone, and reflexive on ends. We order $R$ lexicographically, and let $\pi_0:R \to K$, $\pi_1:R \to K$ be projections on the first, and on the second coordinate, respectively. By the assumptions, $\pi_0$ and $\pi_1$ are surjective $\leq$-monotone functions preserving the root, and hence belong to $\mor(\IIb)$. Since
\[
R=\pi_0\circ\pi_1^{-1},
\]
Remark~\ref{re:FraisseDigraphs} shows that $R$ is an $\IIb$-digraph.
\end{proof}

We define  $\JJub \subseteq \IIub$ as the full subcategory of all $\leq$-monotone $\IIb$-digraphs. 
 
\begin{proposition}
\label{pr:IisDirected}
$\JJub$ is directed, and for every $A \in \IIub$ there is $B \in \JJub$  such that  $B \subseteq A$. In particular, $\JJub$ is co-initial in $\IIub$.
\end{proposition}

\begin{proof}
For $A, B \in \JJub$, let $K=K_A \sqcup K_B$ be ordered by the extension of the linear orders on $K_A$, $K_B$ to a linear order so that  $x\leq x'$ for every $x \in K_A$, $x' \in K_B$. Put $C=A \sqcup B$. Then $C \in \JJub$, and the contractions $p_A:K \to K_A$, $p_B:K \to K_B$ that map all elements in $K_B$,  $K_A$ to an appropriate end, are in $\JJub^A_C$, $\JJub^B_C$, respectively.  Thus, $\JJub$ is directed.

Now fix $A \in \IIub$, and $\sqsp \in \IIb$ such that $A=\sqsp \circ \preceq$. By Corollary \ref{co:CharCMon}(2), there are functions $\sqsp', \preceq' \in \mor(\IIb)$ such that $\sqsp' \subseteq \sqsp$, $\preceq' \subseteq \preceq$. For $B=\sqsp' \circ \preceq'$, Proposition \ref{pr:CharI} gives that $B \in \JJub$, and clearly $B \subseteq A$, i.e., $\Id_K \in \IIub^A_B$, where $K=K_A=K_B$.

\end{proof}

Now we set about proving that $\JJub$ has weak m-amalgamation. Let $K \in \GGRb$, and $\leftarrow \subseteq K \times K$ be a relation. We say that $x \in K$ is a \emph{transshipment} point of $\la$ if there exists a bi-infinite $\la$-walk $(x_n)_{n \in \ZZ}$ such that $x_n=x$ iff $n=0$. Otherwise, $x$ is a \emph{terminal} point. 
\begin{remark}
If $\la \in \IIub$, and $\la$ is $\leq$-monotone, then $x \in K_\la$ is a transshipment point of $\la$ iff there are $x',x'' \in K$ such that $x' < x < x''$ and either $x' \leftarrow x \leftarrow x''$ or $x' \rightarrow x \rightarrow x''$.
\end{remark}
We say that $\leftarrow$ is \emph{simple} if the only terminal points are the ends. It is \emph{very simple} if, except for the ends, every vertex has exactly one incoming arrow and exactly one outgoing arrow. In particular, if $\la$ is very simple, then it is simple. 
Let $x_0, \ldots, x_n$ be the increasing enumeration of all terminal points of $\leftarrow$, and let $$\leftarrow_i=\leftarrow \upharpoonright [x_i,x_{i+1}] \times [x_i,x_{i+1}],$$  $i<n$. We call $\leftarrow_i$'s \emph{pieces} of $\la$. 
The full subcategory of $\JJub$ consisting of all simple $\leq$-monotone $\IIb$-digraphs is denoted by $\JJub_s$, the full subcategory of very simple $\IIb$-digraphs is denoted by $\JJub_v$, and the full subcategory of $\IIb$-digraphs with very simple pieces  is denoted by $\JJub_{pv}$.

By an \emph{orbit} of $\leftarrow \in \JJub_v$ we mean a maximal directed $\la$-path (equivalently: a directed $\la$-path contaning both ends of $\la_K$). Let us denote by $\oo_\leftarrow \in \NN$ the number of orbits of $\leftarrow$, and by $\sss_\leftarrow \in \NN$ the size of the largest orbit.  For a non-end $x \in K_\la$, let $\OO(x)$ be the unique orbit that $x$ belongs to. Let $\OO_1(x)$ be the unique $x' \in K_\la$ such that $x' \la x$, and let $\OO_{n+1}(x)=\OO_1(\OO_n(x))$, if $\OO_n(x)$ is a non-end, and  $\OO_{n+1}(x)=\OO_n(x)$, otherwise.

\begin{remark}
\label{re:TerminalReflexive}
Note that, for $\la \in \JJub_{pv}$, and $x \in K_\la$, we have that $x$ is a terminal point iff $x$ is reflexive.
\end{remark}

\begin{remark}
\label{re:RGBArc}
Let $\la,\tla \in \JJub_v$ be such that $|K_\la|>1$, and let $\sqsp \in \JJub^\la_\tla$. It is straightforward to observe that we can write $K_\tla$ as $3$ intervals $[x_0,x_1] \cup (x_1, x_2) \cup [x_2,x_3]$, where $x_i \leq x_{i+1}$, and
\begin{enumerate}
	\item $\rt \sqsp x$, for $x \in [x_0,x_1]$,
	 \item $\rt, \max K_\la  \not \sqsp x$, for $x \in (x_1,x_2)$,
	 \item $\max K_\la \sqsp x$, for $x \in [x_2,x_3]$,
	 \item $|\OO(x) \cap (x_1,x_2)|=|K_\la|-2$, for every $x \in K_\tla$.
\end{enumerate}  
\end{remark}

In the next two lemmas, we present two methods of lifting a very simple $\IIb$-digraph. The first one splits an orbit into two orbits by splitting each element of the orbit except for the ends, in this way, increasing the number of orbits without increasing their sizes. The second one splits an end into a number of elements, increasing the sizes of orbits without increasing their number. 

\begin{lemma}
\label{le:SplitOrbit}
Let $\la \JJub_v$, and let $x_0 \la \ldots \la x_n$ be an orbit of $\la$. We can split it into two orbits as follows. We construct $\tla \in \JJub_v$ with $K_\la \subseteq K_\tla$, and a  contraction $p \in \JJub^\la_\tla$, by splitting each of $x_1, \ldots, x_{n-1}$ into two points $x_i, y_i$, so that
\begin{enumerate}
\item $K_\tla=K_\la \sqcup \{y_i\}$,
\item $y_i=x_i+1$, $0<i<n$, or $x_i=y_i+1$, $0< i<n$,
\item $x_0 \la y_1 \la \ldots \la y_{n-1} \la  x_n$ is an orbit of $\tla$,
\item $p(y_i)=x_i$, $0<i <n$.
\end{enumerate}
In particular, $\oo_\tla=\oo_\la+1$.
\end{lemma}

\begin{proof}
Conditions (1)-(4) unequivocally describe the construction.
\end{proof}

\begin{lemma}
	\label{le:SameLengthOfOrbits}
	Let $\leftarrow \in \JJub_v$, and let $t \geq \sss_\la$. We can split an end of $\la$ as follows.  We construct $\twoheadleftarrow \in \JJub_v$ with $K_\la \subseteq K_\tla$, and a contraction $p \in \JJub_\twoheadleftarrow^\leftarrow$ so that  
\begin{enumerate}
	\item $\oo_\twoheadleftarrow=\oo_\leftarrow$,
	\item all the orbits of $\twoheadleftarrow$ have size $t$,
	\item $p(y)=\rt$ for every $y \in K_\tla \setminus K_\la$ or $p(y)=\max K_\la$ for every $y \in K_\tla \setminus K_\la$,  
\end{enumerate}
\end{lemma}

\begin{proof}
Since the two possible orientations are symmetric, we may assume that $\rt \la  \rt+1$; the case that $\rt \ra  \rt+1$ is analogous. Let $\oo=\oo_\la$, $\sss=\sss_\la$. Let $x_0 \la x_1  \la  \ldots \la x_m$ be the orbit $O_0$ of $\leftarrow$ such that $x_0=\rt$, and $x_1=\rt+1$. We can assume that $m>1$.  As $\leftarrow$ is very simple, for every $0<i<m-1$, and every orbit $O \neq O_0$, there must be a unique $x \in O$ with $x_i < x<x_{i+1}$. Therefore the orbit of $x_1$ has size $\sss$, and every orbit has size either $\sss-1$ or $\sss$. 

Now, let $y_0, \ldots,y_n$ be the $\leq$-increasing enumeration of $K_\la$ (i.e., $y_0=\rt$). We observe that $y_1,\ldots,y_{\oo}$ belong to distinct orbits, and
hence represent all the orbits of $\la$.  Moreover, either the orbit of $y_{\oo}$ has size $\sss-1$ or all the orbits of $\la$ have size $\sss$.  Next, we define a path $K=\{ y_0, y'_0, y_1, y_2 \ldots, y_n \}$, where $y'_0 \not \in K_\la$, and $y_0< y'_0<y_1$. Finally, we define  $\twoheadleftarrow \subseteq K \times K$ by putting
\begin{enumerate}
\item $y_0 \twoheadleftarrow y'_0 \twoheadleftarrow y_{\oo}$,
\item $y_0 \not \twoheadleftarrow y_{\oo}$,
\item  $y \twoheadleftarrow y'$ iff $y \la y'$, for other cases. 
\end{enumerate}
Then $\twoheadleftarrow \in \JJub_v$, and the contraction $p: K \to K_\la$ such that $p(y'_0)=y_0$ is in $\JJub^\la_\tla$. Moreover, the size of the orbit of $y_{\oo}$ in $\twoheadleftarrow$ is $1$ larger than in $\la$.  By repeating this procedure sufficiently many times, we can ensure that all the orbits in $\tla$ have size $t$.
\end{proof}

\begin{lemma}
	\label{le:IvHasAP}
	$\JJub_v$ has weak m-amalgamation.
\end{lemma}

\begin{proof}
Let $A \in \JJub_v$.  We may assume that $\rt A \rt+1$; the case that $\rt+1 A \rt$ is analogous. We construct a weak m-amalgamation base $B$ whose orbits all have the same length and whose fibres are sufficiently large. First, we fix $B_0 \in \JJub_v$, and $\nil_0 \in \JJub^{A}_{B_0}$ such that for any $y,y' \in K_{B_0}$ with $\rho(y,y') \leq 2$ there is $x \in K_{A}$ such that $x \nil_0 y,y'$. We may also assume that $|K_{B_0}| \geq 4$.  Using Lemma \ref{le:SplitOrbit} and Lemma \ref{le:SameLengthOfOrbits}, we construct $B \in \JJub_v$, and a function $\nil_1 \in \JJub^{B_0}_{B}$ such that all the orbits of $B$ have size $\sss=\sss_{B_0}$, and $|\{y\}^{\nil_1} |>\sss$, for every $y \in K_{B_0}$. In particular, $\oo_{B}>1$. Let $\nil=\nil_0 \circ \nil_1$.

The following simple observation is the key to the proof. Suppose $C \in \JJub_v$, $\sqsp \in \JJb^{B}_{C}$, $z C z'$, and $y, y' \in K_{B}$ are such that $y>\rt$, and $$y \sqsp z, \quad y' \sqsp z'.$$ 
Since every non-terminal point of $B$ has a unique successor, either $y=\max K_ B$, or there is a unique point $y'' \in K_B$ such that $yBy''$. In the latter case, $\rho(y',y'') \leq 1$. Iterating this observation, we get that for any $z \in K_C$, $n \in \NN$, and $y, y' \in K_{B}$ such that $y> \rt$, and $$y \sqsp z, \quad y' \sqsp \OO_n(z),$$ we have that $\rho(y',\OO_n(y)) \leq n$. In particular, since $\OO_\sss(y)=\max K_{B}$, i.e., an orbit of B reaches its terminal point in at most $\sss$ steps, $$\rho(y',\OO_n(y)) \leq \sss.$$

Now, let  ${C_i} \in \JJub_v$, $\sqsp_i \in \JJub^{B}_{C_i}$, $i=0,1$.  Let $y_0=\rt, \ldots, y_n=\max K_{B}$ be the increasing enumeration of $K_{B}$. We note that every orbit of $B$ meets  $\{ y_1, \ldots, y_{\oo_{B}} \}$ exactly once. By Lemmas \ref{le:SplitOrbit} and \ref{le:SameLengthOfOrbits}, we can assume that $C_0=C_1=D$, and $$|[\rt,y_i]^{\sqsp_0}|=|[\rt,y_i]^{\sqsp_1}|,$$ for $i \leq \oo_{B}$. 
%
%
The equalities above imply that, for every $z \in K_D$, there is $w \in \OO(z)$, and $0 \leq i \leq \oo_{B}$, such that $$y_i \sqsp_0 w, \quad y_i \sqsp_1 w.$$  If $z<w$, then, clearly, $$\rt \sqsp_0 z, \quad \rt \sqsp_1 z,$$
$$\rt \nil \circ \sqsp_0 z, \quad \rt \nil \circ \sqsp_1 z.$$
Otherwise, by the above claim applied to $w$ and the unique $n \in \NN$ such that $z=\OO_n(w)$, and because $|\{y\}^{\nil_1} |>\sss$, for $y \in K_{B_0}$, 
there are $y_0,y_1 \in K_{B_0}$ such that $\rho(y_0,y_1) \leq 2$, and $$y_0 \nil_1 \circ \sqsp_0 z, \quad y_1 \nil_1 \circ \sqsp_1 z.$$
Finally, by the choice of $\nil_0$, there is $x \in K_A$ such that $$x \nil \circ \sqsp_0 z, \quad x \nil \circ \sqsp_1 z.$$ 
Since $z \in K_D$ was arbitrary, the intersection  $$\nil \circ \sqsp_0 \cap \nil \circ \sqsp_1$$ is co-surjective. Hence $D$ is the desired m-amalgam.
\end{proof}

\begin{lemma}
	\label{le:TerminalAreStable}
	Let $\leftarrow \in \JJub_{pv}$ be such that each of its very simple pieces has more than one orbit. Let ${\twoheadleftarrow} \in \JJub$, and $\sqsp \in \underline{\mathbf{I}}_\twoheadleftarrow^\leftarrow$. If $y \in K_\tla$ is a terminal point of $\tla$, then there is a terminal point $x \in K_\la$ of $\la$ such that $x \sqsp y$.    
\end{lemma}

\begin{proof}
	Fix a terminal point $y \in K_\twoheadleftarrow$, and  $x \in K_\leftarrow$ such that $x \sqsp y$. Suppose that $x$ is a transshipment point. Since terminal points of
	$\leq$-monotone $\IIb$-digraphs are reflexive, while transshipment
	points of $\IIb$-digraphs with very simple pieces are not, there must
	be $x'\neq x$ adjacent to $x$ such that $x'\sqsp y$, and either $x'$
	is terminal, $x'\la x$, or $x\la x'$. But every very simple piece of $\la$ has more than one orbit, so only the first option is possible, i.e., $x'$ is  a terminal point.
\end{proof}

\begin{lemma}
	\label{le:Jp_Amalg}
	$\JJub_{pv}$ has weak m-amalgamation.
\end{lemma}

\begin{proof}
	Fix  ${\leftarrow_i} \in \JJub_{pv}$, $i=0,1,2$, and $\sqsp_i \in \JJub_{\leftarrow_i}^{\leftarrow_0}$, $i=1,2$. Using Lemma \ref{le:TerminalAreStable}, we can present each $K_{\leftarrow_i}$ as intervals $K_i^j=[x^j,y^j]$, $j \leq m$, $x^j \leq y^j \leq x^{j+1}$, whose ends are terminal points, so that, for every $j \leq m$, all three $\leftarrow_i \upharpoonright K_i^j$, $i=0,1,2$, are very simple or $K_0^j$ is a singleton, and $(K^j_i)^{\sqsubset_i}=K^j_0$, $i=1,2$. Then, for every $j \leq m$, we can amalgamate $\leftarrow_i \upharpoonright K_i^j$, $i=1,2$ over $\leftarrow_i \upharpoonright K_0^j$, using Proposition \ref{pr:IisDirected} or Lemma \ref{le:IvHasAP}, and glue together the obtained amalgams.
\end{proof}

\begin{lemma}
	\label{le:SimpleToVerySimple}
	$\JJub_{pv}$ is cofinal in $\JJub$.
\end{lemma}

\begin{proof}
Since every $\la \in \JJub$ consists of simple pieces, it suffices to show that $\JJub_v$ is cofinal in $\JJub_s$. For $K \in \IIb$, $\leftarrow \subseteq K \times K$ and $x \in K$, let $\ttt_{\la \, x}$, $\ttt_{x \, \la}$ be the sizes of the shortest directed $\la$-path, $\ra$-path, respectively, starting with $x$, and ending with a terminal point. Fix $\leftarrow \in \JJub_s$. Observe that if $|\{x\}^\ra|=1$ for every $x \in K_\leftarrow$ with $\ttt_{\leftarrow \, x}>1$, and $|\{x\}^\la|=1$ for every $x \in K_\leftarrow$ with $\ttt_{x \,\la }>1$, then $\la \in \JJub_v$. Suppose that the latter condition fails. Choose, among its violations, a transshipment point $y \in K_\leftarrow$ with maximal $\ttt_{y \la}$, and $x<x' \in K_\leftarrow$ such that $$x,x' \leftarrow y,$$ $$x'' \not \leftarrow y \mbox{ if } x<x''<x'.$$ As $y$ is a transshipment point, there is $z>y$ with $y \leftarrow z$. We extend $K_\leftarrow$ to a path $K$ by splitting $y$ into two consecutive points $y, y'$ so that $y'$ is placed right above $y$. 
Then we define $\twoheadleftarrow$ on $K$ by modifying $\leftarrow$ in the following way:
	
	\begin{enumerate}
		\item $w \twoheadleftarrow y'$ iff $w \leftarrow y$  and $w \geq x'$,
		\item $y' \twoheadleftarrow w$ iff $y \leftarrow w$,
		\item $w \not\tla y$ iff $w \la y$ and $w \geq x'$,
		\item $w \twoheadleftarrow w'$ iff $w \leftarrow w'$, for other cases. 
	\end{enumerate}
	
	Clearly, ${\tla} \in \JJub_s$, and the contraction $p:K \rightarrow K_\la$ such that $p(y')=y$ is in $\JJub^\la_\tla$. Moreover, $$|\{y\}^\twoheadrightarrow|, \, |\{y'\}^\twoheadrightarrow|< |\{y\}^\ra|,$$ and 
	$$\ttt_{y' \tla } =\ttt_{y \la},$$
	$$ \ttt_{z \tla}=\ttt_{z \la}, \, z \in K_\la.$$ 
Repeating this construction finitely many times eliminates all violations of the above criterion. Repeating the same argument for $\ra$ yields a very simple lift.
\end{proof}

\begin{theorem}
	\label{th:IhasCAP}
	$\IIub$ has weak m-amalgamation. In particular,  $\Homeo_+([0,1])$ has a comeager conjugacy class.
\end{theorem}

\begin{proof}
By Proposition \ref{pr:IisDirected}, $\JJub$ is cofinal in $\IIub$, and, by Lemma \ref{le:SimpleToVerySimple}, $\JJub_{pv}$ is cofinal in $\JJub$. By Lemma \ref{le:Jp_Amalg}, $\JJub_{pv}$ has weak m-amalgamation, so $\IIub$ has weak m-amalgamation.  Finally, by Proposition \ref{pr:IisDirected}, Theorem \ref{th:ConjugacyForFraisse}(2) and Proposition \ref{co:G_IDense}, $\Homeo_+([0,1])$ has a comeager conjugacy class.
\end{proof}

\subsection{The Cantor fan.} 

The \emph{Cantor fan} is the compact space $C = (2^\NN×\times I)/ \sim $, where $2^\NN$ is the Cantor space, and the equivalence relation $\sim$ is defined by $(x,t) \sim (y,s)$ if and only if the pairs are equal or $s= t= 0$. 

A graph $F \in \GGRb$ is a \emph{fan} if it is a tree (i.e., an acylic, connected reflexive graph) with a single vertex of degree strictly larger than $2$. We call this distinguished vertex \emph{the root}, and, somewhat abusing notation, always denote it by $\rt$. The path metric on a fan is always denoted by $\rho$, and $\leq$ is the unique partial ordering compatible with the graph relation, and with $\rt$ as the smallest element. The maximal elements are referred to as \emph{ends}. A \emph{spoke} in a fan $F$ is a path $[\rt,f]$,  where $f \in F$ is an end. The family of all spokes in $F$ is denoted by $*F$. A fan is \emph{regular} if all its spokes have the same size. For fans $F, G$, a morphism ${\sqsp} \in  \GGRb^F_G$ is called \emph{spoke-monotone} if for any $\tau \in *G$, and any path $C \subseteq F$, $C^\sqsp \cap \tau$ is a path. A relation $A \subseteq F \times G$ is \emph{spoke-preserving} if $\tau^A \in *F$, for every $\tau \in *G$.

The category $\CCb \subseteq \GGRb$ consists of all fans and spoke-monotone morphisms $\sqsp \in \GGRb^F_G$ that preserve roots and ends (i.e., $\rt$ and every end $g \in G$ are $\sqsp$-injective, and $\pi_\sqsp(\rt)=\rt$, $\pi_\sqsp(g)$ is an end). In particular, morphisms in $\CCb$ are spoke-preserving. By \cite[Proposition 5.29]{BaBiVi2} and \cite[Theorem 5.39]{BaBiVi2}, $\CCb$ is a \fra \ category, and its \fra \ limit is the Cantor fan. We associate with $\CCb$ a subcategory $\It$ given be a fixed \fra \ sequence.

 



%
\begin{remark}
For $F,G \in \CCb$, $\sqsp \in \GGRb^F_G$, we have $\sqsp \in \CCb^F_G$ iff $\sqsp$ is spoke-preserving, and $\sqsp \uhr \tau^\sqsb \times \tau \in \mor(\IIb)$, for every $\tau \in *G$. 
\end{remark}
%



 

\begin{remark}
For every $A \in \CCub$, we have that $\rt A \rt$, and for every end $x \in K_A$ there is an end $x' \in K_A$ such that $x A x'$.
\end{remark}  
 
\begin{proposition}
\label{pr:FAnDirected}
$\CCub$ is directed.
\end{proposition}

\begin{proof}
Let $A,B \in \CCub$. We can assume that $K_A \cap K_B=\{\rt\}$. Then $C=A \cup B$ is as required.
\end{proof}

For $A \in\CCub$, the \emph{spoke-graph} of $A$ is a directed graph on $*K_A$ defined by putting an arrow from $\sigma'$ to $\sigma$ if $x' A x$ for the ends $x \in \sigma$, $x' \in \sigma'$.  Let $\AAb \subseteq \GGRb$ be the full subcategory of all finite anticliques, i.e., reflexive graphs with no edges other than the loops. Note that morphisms in $\AAb$ are all surjective functions between finite anticliques. In particular, $\AAub$ is the category of all surjective and co-surjective relations on finite anticliques. As in \cite{Kw}, by a \emph{spiral} we mean a directed graph consisting of two directed cycles (possibly consisting of only one element, i.e., a loop), and a single directed path connecting them.  Let $\SSub \subseteq \AAub$ be the full subcategory consisting of all $\AAb$-digraphs that are disjoint unions of spirals. It has been proved in \cite{Kw} that $\SSub$ is co-initial in $\AAub$ (Lemma 4.8), and is a \fra \ category (Theorem 3.1). As a matter of fact,  a stronger notion of morphism of $\AAb$-digraphs, called \emph{epimorphism}, is considered there: epimorphisms are arrow-preserving (not just weakly arrow-preserving), and arrow-surjective. However, for functions, arrow-preservation is equivalent to weak arrow-preservation, and it is routine to verify (see, e.g., \cite[Proposition 3.3]{Kw} and its proof) that morphisms in $\SSub$ are always arrow-surjective, i.e., they are epimorphisms.

\begin{remark}
\label{re:SpiralStructure}
 Every vertex in a spiral has either exactly one incoming, and one outgoing arrow, one incoming, and two outgoing arrows or two incoming, and one outgoing arrows.
\end{remark}

Weak m-amalgamation in $\IIub$, together with amalgamation in $\IIb$, $\CCb$, and $\SSb$, are sufficient to show weak m-amalgamation in $\CCub$. However, the argument we present below uses auxiliary categories $\CCb_2$, and $\CCub_2$ rather than $\IIb$, and $\IIub$. Although this approach is somewhat more tedious, its advantage is that it will later be adapted to establish weak m-amalgamation in a category of digraphs corresponding to the Lelek fan.

As a matter of fact, weak m-amalgamation in $\IIub$, together with amalgamation in $\IIb$, $\CCb$, and $\SSb$, are sufficient to show weak m-amalgamation in $\CCub$ but the argument we present below uses certain auxiliary categories $\CCb_2$, and $\CCub_2$ rather than $\IIb$, and $\IIub$. Although this approach is somewhat more tedious, its advantage is that it can be adapted to establish weak m-amalgamation in a category of digraphs corresponding to the Lelek fan.

Let $\CCb_2 \subseteq \CCb$ be the wide subcategory consisting of all regular fans $F$ with exactly two spokes, denoted by $\sigma_F, \tau_F \in *F$. Let $\CCub_2$ be the category of all surjective and co-surjective $\leq$-monotone relations $A \subseteq \sigma_F \times \tau_F$, $F \in \CCb_2$, with weakly arrow-preserving morphisms from $\CCb_2$ as morphisms;  in this context, $\sigma_A$, $\tau_A$ stand for $\sigma_{K_A}$, $\tau_{K_A}$, respectively.  Suppose that $A \in \CCub$ is such that $K_A$ is a regular fan, and $\sigma \in *K_A$ is such that $\sigma^A \in *K_A$. We define $A_{\sigma}=A \uhr \sigma^A \times \sigma$.



\begin{lemma}
	\label{le:C2Lift}
	Let $F,G \in \CCb_2$,  let $A \in \CCub_2$ be such that $K_A=F$, and let $\sqsp \in (\CCb_2)^F_G$ be a function. There exists a $\sqsp$-lift $B \in \CCub_2$ of $A$. 
\end{lemma}

\begin{proof}
	As $\sqsp$ can be decomposed into functions that split a single point into two points, and the inverse of a $\leq$-monotone relation is $\leq$-monotone, it suffices to consider the case that $\sqsp$ splits a point $y \in \tau_F$ into two points $y, y'=y+1 \in G$. If there is $x \in \sigma_F$ such that $x A y$, $x+1 A y$, we take the $\leq$-maximal such $x$, and define $$B=(A \cup \{ (x+1,y+1)) \} \setminus \{(x+1,y)\}.$$ Otherwise, let $x$ be the unique $x \in K_A$ with $x A y$. We define $$B=A \cup \{ (x,y+1) \}.$$  
	%
It is straighforward to verify that $B$ is as required.
\end{proof}

\begin{lemma}
\label{le:C2Bijective}
For every $A \in \CCub_2$, there exists a function $\sqsp \in (\CCb_2)^{K_A}_\bullet$, and a bijective $\sqsp$-lift $B \in \CCub_2$ of $A$, i.e., $B \subseteq  \sigma_{B} \times \tau_{B}$ is a bijective relation. 
\end{lemma}

\begin{proof}
For a given $y \in K_A$ such that there exist two distinct $x,x' \in K_A$ with $x A y$, $x' A y$, we split $y$ into two points $y, y'$, and define $B$ by modifying $A$ so that $$(x,y), (x',y') \in B, \quad (x',y), (x,y') \not \in B.$$ After repeating this construction sufficiently many times, and considering also $A^{-1}$, we will obtain the required $\sqsp$ and $B$. 
\end{proof}
	
\begin{lemma}
\label{le:C2MAmalgamation}
Let   $A, {B_i} \in \CCub_2$ be bijective, and let $\sqsp_i \in (\CCub_2)^{A}_{B_i}$, $i=0,1$. There exist a bijective  $C  \in \CCub_2$, and functions $\sqsp'_i \in (\CCub_2)^{B_i}_{C}$, $i=0,1$, that m-amalgamate $B_0$, $B_1$ over $A$. 
\end{lemma}

\begin{proof}
 Let $\{x_i\}$, $\{y_i\}$, $i \leq n$, be the increasing enumerations of $\sigma_A$, $\tau_A$, i.e., $x_i A y_i$, $i \leq n$. For every $i\leq n$, $B \in \CCub_2$, and $\sqsp \in (\CCub_2)^A_B$,  define
 \[
 N(\sqsp,i)=
 \left|
 \left\{ (c,d)\in B: \exists j \leq i \, (x_j \sqsp_i c \text{ and }  y_j \sqsp_i d) \right\} \right|.
 \]

We show that there exist $C \in \CCub_2$, and functions $\sqsp'_i \in (\CCub_2)^{B_i}_{C}$, $i=0,1$, such that  
$$N(\sqsp_0 \circ \sqsp'_0,i)=N(\sqsp_1 \circ \sqsp'_1,i),$$
%
for $i \leq n$.  We leave it to the reader to verify that such $C$ m-amalgamates $B_0$, $B_1$ over $A$.  

Suppose that for some $i\leq n$, and all $j<i$, 
$$ N(\sqsp_0,j)=N(\sqsp_1,j), $$
and
$$ N(\sqsp_0,i)<N(\sqsp_1,i). $$
We fix a $\sqsp_0$-co-injective $d \in \tau_{B_0}$ such that $y_i\sqsp_0 d$, and the unique $c \in \sigma_{B_0} $ with $c B_0 d$. Note that $c$ must also be $\sqsp_0$-co-injective, and $x_i\sqsp_0 c$. Next, we construct a bijective $B'_0$ from $B_0$ by splitting $c$ and $d$ into $c, c'$ and $d, d'$, and putting $c' B'_0 d'$. Finally, we define the function $\sqsp'_0$ so that it splits $c$ into $c$, $c'$, and $d$ into $d$, $d'$.
Then
 $$  N(\sqsp_0,i)<N(\sqsp_0 \circ \sqsp'_0,i) \leq N(\sqsp_1,i).$$
 If the reverse inequality holds, we apply the same construction to
 $B_1$.
 
Repeating the appropriate splitting until equality is obtained, and then proceeding successively for $i=0,\ldots,n$, we obtain bijective $B'_0, B'_1 \in \CCub_2$, and functions $$\sqsp'_0 \in (\CCb_2)^{K_{B_0}}_{K_{B'_0}}, \quad \sqsp'_1 \in (\CCb_2)^{K_{B_1}}_{K_{B'_1}}$$ such that
 $$N(\sqsp_0 \circ \sqsp'_0,i)=N(\sqsp_1 \circ \sqsp'_1,i),$$
 for $i \leq n$. In particular, $K_{B'_0}$, and $K_{B'_1}$ have all the spokes of the same size, so, by bijectivity, we can assume that $B'_0=B'_1=C$.  
\end{proof}

Lemmas \ref{le:C2Bijective} and \ref{le:C2MAmalgamation} immediately imply:

\begin{corollary}
\label{co:C2MAmalgamation}
$\CCub_2$ has weak m-amalgamation, and the amalgamating morphisms can always be chosen to  be functions.
\end{corollary}

\begin{proposition}
	\label{pr:CatMCoInitial}
	The full subcategory $\DDub \subseteq \CCub$ of all $A \in \CCub$ that satisfy the following conditions:
	\begin{enumerate}
		\item $K_{A}$ is regular,
		\item $A$ is spoke-preserving, and $A_{\sigma} \in \CCub_2$, for every  $\sigma \in *K_A$,
		\item the spoke-graph of $A$ is in $\SSb$.
	\end{enumerate}
	is co-initial in $\CCub$.
\end{proposition}

\begin{proof}
	Fix $A\in\CCub$, and write $F=K_A$. Choose $G\in\CCb$ and
	$\sqsp_0,\sqsp_1\in\CCb_G^F$ such that $A=\sqsp_0\circ\sqsb_1$. 
	By splitting $\rt$, if necessary, we may assume that $F$ is regular.
	Applying Corollary~\ref{co:CharCMon}(2) to each spoke of $G$, we may
	also assume that $\sqsp_0$ and $\sqsp_1$ are functions.
	
	We construct a lift of $A$ satisfying (1)--(3). First, we separate conflicting spokes. Since $\sqsp_0$ and $\sqsp_1$ are functions, condition~(2) is satisfied
	provided that distinct spokes of $G$ having the same
	$\sqsp_1$-image always have different $\sqsp_0$-images. Thus it is
	enough to eliminate pairs of distinct spokes
	$\tau_0,\tau_1\in *G$ such that
	\[
	\tau_0^{\sqsb_0}=\tau_1^{\sqsb_0},
	\quad
	\tau_0^{\sqsb_1}=\tau_1^{\sqsb_1}.
	\]
	
	Suppose such a pair exists, and write
	\[
	\tau_0^{\sqsb_0}=\tau_1^{\sqsb_0}=\sigma_0,
	\quad
	\tau_0^{\sqsb_1}=\tau_1^{\sqsb_1}=\sigma_1.
	\]
	We choose a spoke $\tau\in *G$ with
	$\tau^{\sqsb_1}=\sigma_0$.
	Form regular fans $F'$ and $G'$ by adding copies
	$\sigma'_0$ of $\sigma_0$ and $\tau'$ of $\tau$, respectively,
	and let
	$p_F:F'\to F$ and $p_G:G'\to G$
	be the corresponding contractions.
	
	We define $\sqsp'_0,\sqsp'_1$ exactly as follows:
	$\sqsp'_0$ copies the restriction of $\sqsp_0$ from
	$\sigma_0\times\tau_1$ to $\sigma'_0\times\tau_1$,
	copies the restriction of $\sqsp_0$ from
	$\tau^{\sqsb_0}\times\tau$ to
	$\tau^{\sqsb_0}\times\tau'$,
	$\sqsp'_1$ copies the restriction of $\sqsp_1$ from
	$\tau^{\sqsb_1}\times\tau$ to
	$\tau^{\sqsb_1}\times\tau'$,
	and elsewhere the two relations agree with
	$\sqsp_0,\sqsp_1$.
	Let $A'=\sqsp'_0\circ\sqsb'_1$.
	
	The contractions $p_F$ and $p_G$ witness that $A'$ is a lift of $A$.
	Moreover, the chosen conflicting pair disappears, while the new spoke
	$\tau'$ cannot create a new conflict because its
	$\sqsp'_1$-image is the new spoke $\sigma'_0$.
	Hence the number of conflicting pairs of spokes is strictly smaller
	than before. Repeating the construction finitely many times yields a
	lift $B'$ of $A$ satisfying (1) and (2).
	
	Finally, we ensure that (3) holds. Let $K$ be the spoke graph of $B'$. Since $\SSb$ is co-initial in $\AAb$, there are $S\in\SSb$ and $\sqsp\in\AAb^K_S$. 
	Lifting every spoke of $B'$ along $\sqsp$ produces a lift $B$ of $B'$ whose spoke graph is $S$. The construction only duplicates spokes and their local relations, so regularity, spoke-preservation, and the property that every local restriction belongs to $\CCub_2$ are preserved. Thus $B\in\DDub$, and $B$ is a lift of $A$.
	
\end{proof}

\begin{lemma}
	\label{le:DAmalg}
	$\DDub$ has weak m-amalgamation.
\end{lemma}

\begin{proof}
	Let $A\in\DDub$. By Lemma \ref{le:C2Bijective} and Corollary
	\ref{co:C2MAmalgamation}, for every $\sigma\in *K_A$ there are
	$B_\sigma\in\CCub_2$ and a function
	\[
	\nil_\sigma\in(\CCub_2)^{A_\sigma}_{B_\sigma}
	\]
	such that $\nil_\sigma$ is a weak $m$-amalgamation base for $A_\sigma$.
	
	Since $\IIb\cap\GGRFb$ has amalgamation
	\cite[Proposition~5.13]{BaBiVi2}, Lemma \ref{le:C2Lift} allows us to
	replace the $B_\sigma$'s by further lifts so that all copies
	corresponding to the same spoke of $K_A$, as well as the corresponding
	restrictions of the maps $\nil_\sigma$, agree. We may also assume that
	all the spokes occurring in the $K_{B_\sigma}$'s have the same size.
	Thus the $B_\sigma$'s glue to a fan $B\in\DDub$, and the functions
	$\nil_\sigma$ glue to a function $\nil\in\DDub^A_B$.
	
	
	Fix $C_i\in\DDub$ and $\sqsp_i\in\DDub^B_{C_i}$, $i=0,1$. Let $S,T_i\in\SSb$ be the spoke-graphs of $B,C_i$, respectively, and
	let $\alpha_i\in\SSb^S_{T_i}$ be induced by $\sqsp_i$. Since $\SSb$ has amalgamation, there are
	$U\in\SSb$ and morphisms $\alpha'_i\in\SSb^{T_i}_U$, $i=0,1$, which amalgamate $\alpha_0$ and $\alpha_1$.
	
	Fix an arrow $u\la v$ in $U$. Let $\sigma_i,\tau_i\in T_i$ be the
	images of $u,v$ under $\alpha'_i$, and let
	$\sigma_B,\tau_B\in S$ be their common images under $\alpha_i$.
	The restriction
	\[
	B\uhr\sigma_B\times\tau_B
	\]
	is one of the local weak $m$-amalgamation bases chosen above.
	Therefore the restrictions
	\[
	C_i\uhr\sigma_i\times\tau_i,\quad i=0,1,
	\]
	admit an $m$-amalgam over the corresponding restriction of $A$.
	Thus we obtain paths $\sigma^{u,v},\tau^{u,v}$, a relation
	\[
	D_{u,v}\subseteq\sigma^{u,v}\times\tau^{u,v},
	\quad D_{u,v}\in\CCub_2,
	\]
	and functional morphisms
	\[
	\sqsp^{u,v}_{i,\sigma}\subseteq
	\sigma_i\times\sigma^{u,v},
	\quad
	\sqsp^{u,v}_{i,\tau}\subseteq
	\tau_i\times\tau^{u,v},
	\quad i=0,1,
	\]
	that weakly m-amalgamate $C_i \uhr \sigma_i \times \tau_i$ over the corresponding
	restriction of $A$.
	
	By taking further lifts with Lemma \ref{le:C2Lift}, we may assume that
	all paths $\sigma^{u,v},\tau^{u,v}$ have the same size. We now make
	the local constructions compatible on common spokes. Suppose first that
	\[
	u\la v\la w.
	\]
	For each $i=0,1$, the functions
	\[
	\sqsp^{u,v}_{i,\tau}\colon \tau_i \to \tau^{u,v}
	\quad\text{and}\quad
	\sqsp^{v,w}_{i,\sigma}\colon \tau_i \to \sigma^{v,w}
	\]
	have the same domain. By amalgamation in $\IIb\cap\GGRFb$, there are
	functions
	\[
	\sqni^{u,v}_{\tau,i}\in\IIb^{\tau^{u,v}}_\bullet,
	\quad
	\sqni^{v,w}_{\sigma,i}\in\IIb^{\sigma^{v,w}}_\bullet
	\]
	such that
	\[
	\sqsp^{u,v}_{i,\tau}\circ\sqni^{u,v}_{\tau,i}
	=
	\sqsp^{v,w}_{i,\sigma}\circ\sqni^{v,w}_{\sigma,i}.
	\]
	Thus, for each $i$, the maps induced by the two adjacent local
	amalgams agree after passing to suitable functional extensions. We next amalgamate, again in $\IIb\cap\GGRFb$, all four functions $\sqni^{u,v}_{\tau,i}, \sqni^{v,w}_{\sigma,i}$, $i=0,1$,
	over a single path. Applying Lemma \ref{le:C2Lift} to
	$D_{u,v}$ and $D_{v,w}$ along the resulting functional extensions, we
	may replace these local amalgams by further lifts for which
	\[
	\tau^v=\tau^{u,v}=\sigma^{v,w}
	\]
	and, for $i=0,1$,
	\[
	\sqsp_i^v
	=
	\sqsp^{u,v}_{i,\tau}
	=
	\sqsp^{v,w}_{i,\sigma}.
	\]
	In other words, the two adjacent local amalgams now use literally the
	same copy of the spoke corresponding to $v$, and the maps from both
	$C_0$ and $C_1$ agree on that spoke.
	
	
	The same finite-amalgamation argument applies at branching vertices, where three local amalgams meet at the same spoke. By Remark \ref{re:SpiralStructure}, these are all possible local configurations
	in a spiral. Repeating the construction at every vertex of $U$, we obtain, for
	each $v\in U$, a single path $\tau^v$ and functions $\sqsp_i^v\colon \tau_i \to \tau^v$, $i=0,1$, such that every local relation $D_{u,v}$ is defined on
	$\tau^u\times\tau^v$ and its local morphisms are precisely the
	restrictions of $\sqsp_i^u\cup\sqsp_i^v$. Notice that every
	synchronization is obtained by replacing the relevant local
	$m$-amalgams by common lifts. Since postcomposition with the lifting maps preserves co-surjectivity, the local $m$-amalgamation property is preserved throughout the construction.
	
	Glue the paths $\tau^v$, $v\in U$, at their roots and put $D=\bigcup_{u\la v}D_{u,v}$.
	The compatibility just established makes this union well defined.
	Moreover, $D$ is spoke-preserving, every two-spoke restriction belongs
	to $\CCub_2$, $K_D$ is regular, and the spoke-graph of $D$ is
	$U\in\SSb$. Thus $D\in\DDub$.
	
	For $i=0,1$, the maps $\sqsp_i^v$, $v\in U$, glue to a functional
	morphism
	\[
	\sqsp'_i\in\DDub^{C_i}_D.
	\]
	For every arrow $u\la v$ in $U$, the restriction of
	\[
	\bigl(\nil\circ\sqsp_0\circ\sqsp'_0\bigr)
	\cap
	\bigl(\nil\circ\sqsp_1\circ\sqsp'_1\bigr)
	\]
	to the corresponding two-spoke piece is co-surjective by the choice
	of $D_{u,v}$. Every point of $K_D$ belongs to one of these local
	two-spoke amalgams, since every spoke of a spiral is incident with an
	arrow. Likewise, every point of $K_A$ belongs to the corresponding
	restriction of $A$. Hence the union of the local co-surjective
	intersections is co-surjective, and therefore the whole intersection
	above is co-surjective in $K_A\times K_D$. Thus
	$D,\sqsp'_0,\sqsp'_1$ $m$-amalgamate $C_0,C_1$ over $A$.
\end{proof}

\begin{lemma}
\label{le:C_IDense}
$G_\CCb$ is a dense subset of $\Homeo(C)$.
\end{lemma}

\begin{proof}
Fix  $k,l \in \NN$, and $\sqsp \in (\KKb_C)^k_l$. Fix  $m \in \NN$, and $\sqsp' \in (\KKb_C)^l_m$ such that $\sqsp \circ \sqsp' \subseteq \succeq^k_m$. Fix $\sigma \in *K_k$, $\tau \in *K_l$ such that $\tau^\sqsb=\sigma$, and $\upsilon \in *K_m$ such that $\upsilon^\sqsb=\tau$. Let $\sqsp_\tau=\sqsp \uhr \sigma \times \tau$. We note that $\preceq^k_m \uhr \sigma \times \upsilon$ is monotone as a composition of monotone relations. By Corollary \ref{co:G_IDense}, there is $\sqni_\tau \in \IIb$ such that $\sqsp_\tau \subseteq \sqni_\tau$. Put $\sqni=\bigcup_\tau \sqni_\tau$. As $ \pi_{\sqni} (\rt)=\rt$, we have $\sqni \in \mor(\CCb)$. Now we get the statement of the lemma from Proposition \ref{pr:Robust}(3).   
\end{proof}

From Proposition \ref{pr:FAnDirected}, Lemma \ref{le:DAmalg}, Lemma \ref{le:C_IDense}, and Theorem \ref{th:ConjugacyForFraisse}(2) we immediately get
\begin{theorem}
	$\CCub$  has weak m-amalgamation, and $\Homeo(C)$ has a comeager conjugacy class.
\end{theorem}

Finally, we study topological entropy of the generic homeomorphism of the Cantor fan. We define transshipment and terminal points in $\CCb$-digraphs in the same way as for the category $\IIb$:  for $\la \in \CCub$, $x \in K_\la$ is a transhipment point if there exists a bi-infinite $\la$-walk $(x_n)_{n \in \ZZ}$ such that $x_n=x$ iff $n=0$; otherwise it is a terminal point.  For a spiral $S \in \SSb$ with cycles $R$, $B$ and directed path $G$ from $R$ to $B$, arrows in $R$ are called \emph{red arrows}, in $G$ -- \emph{green arrows}, and in $B$  -- \emph{blue arrows}. We refer to the ends of the path $G$ as the red and the blue branching points, depending on whether they also belong to $R$ or to $B$.

\begin{remark}
\label{re:TerminalCycle}
	If $x$ is a terminal point, then there is a $\la$-cycle $O$ such that $x \in O$. 
\end{remark}

\begin{proposition}
	\label{pr:TopEntropy}
	Let $\leftarrow \in \DDub$ be such that its spoke graph is a single cycle of size $n$. There is $k_0 \in \NN$ such that for every $k \geq k_0$ there is a lift $\tla \in \MMub$ of $\la$ satisfying the following properties:
	\begin{enumerate}
		\item $K_\tla$ is a regular graph with spokes of size $k$, and the spoke graph of $\twoheadleftarrow$ is a cycle of size $n$,
		\item every transshipment point has exactly one incoming and one outgoing arrow,
		\item any two distinct cycles are disjoint,
		\item $x$ is a terminal point iff there is a (unique) cycle $O$ such that $x \in O$,
	\end{enumerate}
	In particular, the full subcategory $\DDub_0 \subseteq \DDub$ of $\CCb$-digraphs such that the cycles forming the spirals of their spoke graph satisfy (1)-(4) is co-initial in $\CCub$.
\end{proposition}

\begin{proof}
	By Condition (2) of Proposition \ref{pr:CatMCoInitial}, $\la$ is $\leq$-monotone with respect to the partial ordering $\leq$ on $K_\la$.  Therefore exactly the same construction as in the proof of Proposition \ref{le:SimpleToVerySimple} gives a lift $\tla$ of $\la$ that satisfies (1) and (2) with sizes of spokes equal to some $k_0 \in \NN$. To obtain such $\tla$ with spokes of size $k>k_0$, we can additionally split cycles.
	
	We show (3). Suppose that $c_0 \la \ldots \la c_m$,  and $d_0 \la  \ldots \la d_n$ are cycles (i.e. $c_0 \la c_m$, $d_0 \la d_n$) that are not disjoint. Without loss of generality, we can assume that $m=n$, and $\pi_S(c_i)=\pi_S(d_i)$, $i \leq n$, and that each $d_i$ is an immediate successor of $c_i$ with regard to $\leq$. There exists (possibly, after a renumeration of the cycle) $0<i<n$ such that $c_{i-1}<d_{i-1}$, $c_i=d_i$. We construct a lift $\tla$ of $\la$ in the following way. We split $c_i$ into $c_i$, $c'_i$, and define $\tla$ so that $$c_{i-1} \tla c'_i \tla c_{i+1}, \quad c_{i-1} \not \tla c_i \not \tla c_{i+1}$$ and $\tla$ agrees with $\la$ for other cases. At each step the number of pairs of distinct non-disjoint cycles
	strictly decreases, while no new such pair is created. Since there are
	only finitely many cycles, the process terminates. Applying the above
	construction repeatedly, we therefore obtain a lift $\tla$ of $\la$
	satisfying~(3).
	

Now we show that (2) and (3) imply (4). The implication from left to right follows from Remark~\ref{re:TerminalCycle} and (3). Conversely, suppose that a cycle $O$ contains a transshipment point $x$, and let $(x_i)_{i\in\mathbb Z}$ be a bi-infinite $\tla$-walk such that $x_0=x$ and $x_i\neq x$ for $i\neq0$. By (2), $O$ must contain a terminal point; let $c$ be the first such point along $O$ starting from $x$. Since all preceding points are transshipment, (2) forces $(x_i)$ to follow $O$ from $x$ to $c$. As $c$ is terminal, it occurs again in $(x_i)$. The segment between two occurrences of $c$ contains a cycle through $c$, which by (3) must be $O$, and hence contains $x$. This contradicts the choice of $(x_i)$. Thus every point belonging to a cycle is terminal.

To prove the ``in particular'' part, consider first $\la\in\DDub$ whose spoke-graph is a single spiral, with red and blue cycles $R,B$ and green path $G$. Apply the above construction separately to the restrictions of $\la$ to $R$ and $B$, choosing the same spoke size. For the restriction to $G$, temporarily regard the red and blue branching points as reflexive, so that they form trivial cycles, and apply the same construction. Finally, glue the resulting lifts. Applying this independently to all spirals gives the required lift in $\DDub_0$.

\end{proof} 

\begin{theorem}
A generic homeomorphism of $C$ has no Li--Yorke pairs. In particular, zero topological entropy is generic for homeomorphisms of $C$.
\end{theorem}

\begin{proof}
First, we construct a \fra \ sequence $(K_n, \sqsp_n)$ in $\CCb$, so that, for every $m \in \NN$, and every $\la \in \CCub$ with $K_\la=K_m$ there exist $n \geq m$, and $\twoheadleftarrow \in \DDub_0$ such that $K_\twoheadleftarrow=K_n$, and $\twoheadleftarrow$ is a $\sqsp^m_n$-lift of $\la$. This can be easily done using amalgamation in $\CCb$, and in $\DDub$, together with co-initiality of $\DDub_0$ in $\DDub$ given by Proposition \ref{pr:TopEntropy}. In particular, the \fra \ limit of $(K_n, \sqsp_n)$ is the Cantor fan $C$. Moreover, we have that the set $Y \subseteq \Homeo(C)$ of all $\phi \in \Homeo(C)$ such that for every $m \in \NN$ there is $n \geq m$, and $\la \in \DDub_0$ such that $K_\la=K_ n$, and $\phi \in [\la]$, is comeager in $ \Homeo(C)$. Thus it is enough to prove that every $\phi \in Y$ has no Li--Yorke pairs.

Fix $\phi \in Y$. We show that $\phi$ has no Li-Yorke pairs. We fix a compatible metric $d$ on $C$, and $x, y \in C$. We fix $\epsilon>0$, and $\la \in \DDub_0$ such that $\mbox{mesh}(K_\la)<\epsilon$, and $\phi \in [\la]$. Let $(x^\epsilon_n)$, $(y^\epsilon_n)$ be $\la$-walks in $K_\leftarrow$ such that $$\phi^n(x) \in x^\epsilon_n, \quad \phi^n(y) \in y^\epsilon_n,$$ $n \in \NN$. 
 Since every directed path in a spiral eventually enters the blue cycle, there is $m \in \NN$ such that all the arrows $$x_n^\epsilon \la x_{n+1}^\epsilon, \quad y_n^\epsilon \la y_{n+1}^\epsilon$$ are blue whenever $n \geq m$. Thus, by Proposition \ref{pr:TopEntropy}(4),  the sequences $(x^\epsilon_n)$ and $(y^\epsilon_n)$ are eventually periodic. We distinguish two cases according to the eventual position of the periodic parts of the symbolic walks. First, for every $\epsilon>0$ there is $m \in \NN$ such that $x^\epsilon_m$ and $y^\epsilon_m$ are periodic, and $\rho(x^\epsilon_m,y^\epsilon_m) \leq 1$. Then monotonicity of $\la$ implies that
$x^\epsilon_n$ and $y^\epsilon_n$ remain in the same spoke, and satisfy
$\rho(x^\epsilon_n,y^\epsilon_n)\le 1$ for every $n\ge m$. Since $\mesh(K_\la)<\epsilon$, it follows that $$\lim d(\phi^n(x),\phi^n(y))=0.$$ Otherwise, there is $\epsilon>0$ such that $\rho(x^\epsilon_n,y^\epsilon_n) \geq 2$ for all sufficiently large $n$. The corresponding cells have disjoint closures, and, since the two symbolic walks are eventually periodic, only finitely many such pairs occur. Hence they are uniformly separated by a positive distance, and therefore $$\liminf d(\phi^n(x),\phi^n(y))>0.$$ In either case, $x$, $y$ are not a Li--Yorke pair.
\end{proof}

\subsection{The Lelek fan.} 

 A Lelek fan is a nondegenerate subcontinuum $L$ of the Cantor fan whose ends (in the topological sense) are dense in $L$. All Lelek fans are homeomorphic, and it therefore makes sense to speak of \emph{the} Lelek fan. See, e.g., \cite{BaBiVi2} for more details on this topic. The category $\LLb \subseteq \GGRb$ controlling the behavior of the Lelek fan consists of all fans, i.e., $\ob(\LLb)=\ob(\CCb)$, and spoke-monotone morphisms that preserve roots. By \cite[Proposition 5.29]{BaBiVi2} and \cite[Theorem 5.43]{BaBiVi2}, $\LLb$ is a \fra \ category, and its \fra \ limit is the Lelek fan. 
 
It turns out that the category $\LLub$ is very similar to $\CCub$. However, when working with the former, one needs to deal with two interrelated technical difficulties: morphisms are not necessarily spoke-preserving and, when restricted to spokes, they need not even be co-injective. Nevertheless, the line of reasoning from the previous section applies to the Lelek fan as well, either verbatim or with only minor adjustments. Unfortunately, we are not aware of any general transfer principle in the spirit of \cite[Section~3.1]{BaKuKwMa} that would account for these differences. We therefore revisit the arguments from the Cantor fan section, indicating only the points where such adjustments are required and omitting those parts of the proofs that remain unchanged.
 
 \begin{remark}
 	Let $F, G \in \LLb$, ${\sqsp} \subseteq \LLb^F_G$.
 	\begin{enumerate}
 		\item Let $\tau \in *G$. As $\sqsp$ preserves roots, we have that $\tau^\sqsb=[\rt,f]$ for some $f \in F$. 
 		\item Let $\sigma \in *F$. By co-injectivity and spoke-monotonicity of $\sqsp$, there exists $\tau \in *G$ such that $\tau^\sqsb=\sigma$, and $\sqsp \uhr \sigma \times \tau \in \mor(\IIb)$.
 	\end{enumerate}
 \end{remark}
 
 \begin{remark}
\label{re:Break}
 For any $\sqsp \in \LLb^F_G$, and $\tau \in *G$, we have that $\sqsp \uhr \tau^\sqsb \times \tau$ is co-surjective, and monotone. However, because $\sqsp \uhr \tau^\sqsb \times \tau$ needs not be co-injective, it is not true in general that $\sqsp \uhr \tau^\sqsb \times \tau \in \mor(\IIb)$, even if $\sqsp$ is spoke-preserving. Nevertheless, by \cite[Lemma 5.18]{BaBiVi2}, for any $f \in \tau^\sqsb$ that is not $\leq$-maximal in $\tau^\sqsb$, there exists $g \in \tau$ such that $\pi_\sqsp(g)=f$. In particular, co-injectivity of a spoke-preserving $\sqsp \in \LLb^F_G$ may break down only at the ends of $F$. 
 \end{remark}
 
Before we start investigating $\LLub$, let us consider auxiliary categories needed because of the complications indicated in Remark \ref{re:Break}. Let $\LLb_2 \subseteq \LLb$ be the full subcategory whose objects are objects in $\CCb_2$. Let $\LLb_2^*$ be the category whose objects are objects in $\LLb_2$, but whose morphisms are co-surjective, spoke-monotone relations that preserve the roots. Note that $\LLb_2^* \not \subseteq \GGRb$ because morphisms in $\LLb_2^*$ are not necessarily co-injective (nor surjective), however, as already mentioned, co-injectivity may fail only at the ends. Let $\LLub^*_2$ be the category whose objects are $\leq$-monotone $\LLb_2$-digraphs, and morphisms are weakly arrow-preserving morphisms from  $\LLb_2^*$. We define $\IIb^*$ and $\JJub^*$ analogously, i.e., morphisms in $\IIb^*$ are co-surjective, and monotone relations between paths that preserve minimal elements, and objects in $\JJub^*$ are $\leq$-monotone $\IIb$-digraphs on paths. 

\begin{remark}
For $F, G \in \LLb$, $\sqsp \in \GGRb^F_G$, we have $\sqsp \in \LLb^F_G$ iff, for every $\tau \in *G$,  $\tau^\sqsb=[\rt,f]$ for some $f \in F$, and $\sqsp \uhr \tau^\sqsb \times \tau \in \mor(\IIb^*)$.
\end{remark}

 Exactly the same arguments as in the proofs of the corresponding results on $\IIb$, i.e., Proposition \ref{pr:CharCMon}, Corollary \ref{co:CharCMon}, and Proposition \ref{co:G_IDense}, give:

 \begin{proposition}
 	Let $K, L \in \IIb^*$, and let $\sqsp \subseteq K \times L$ be an edge-preserving relation such that $\pi_{\sqsp}(\rt)=\rt$. Then $\sqsp$ is monotone iff for any $l,l' \in L$, we have that $l \leq l'$ implies $\alpha_{l} \leq \alpha_{l'}$, $\beta_{l} \leq \beta_{l'}$.
 \end{proposition} 
 
 \begin{corollary}
 For every $\sqsp \in \mor(\IIb^*)$ there exists a $\leq$-monotone function $\sqsp' \in \mor(\IIb^*)$ such that $\sqsp' \subseteq \sqsp$.
 \end{corollary} 
 
 
 
 \begin{proposition}
 \label{co:G_I*Dense}
 	Let $K,L,M \in \IIb^*$, and let $\sqsp \in (\IIb^*)^K_L$, $\sqsp' \in (\IIb^*)^L_M$, $\succeq \in (\IIb_2^*)^K_M$ be such that
 	\begin{enumerate}
 		\item $\pi_{\sqsp}(\rt)=\rt$, $\pi_{\sqsp'}(\rt)=\rt$,
 		\item $\sqsp \circ \sqsp' \subseteq \succeq$,
 		\item $\succeq$ is monotone.
 	\end{enumerate}
 	Then there exists a monotone $\sqni \in (\IIb^*)^K_L$ such that $\sqsp \subseteq \sqni$.
 \end{proposition}
 %
  


Because $\ob(\LLub_2^*)=\ob(\CCub_2)$,  and $\mor(\CCub_2) \subseteq \ob(\LLub_2^*)$, Lemmas \ref{le:C2Lift} and \ref{le:C2Bijective} immediately imply the following  two facts.

\begin{lemma}
	\label{le:L2Lift}
	Let $F,G \in \LLb^*_2$,  let $A \in\LLub^*_2$ be such that $K_A=F$, and let $\sqsp \in (\LLb^*_2)^F_G$ be a spoke-preserving function. There exists a $\sqsp$-lift $B \in\LLub^*_2$ of $A$. 
\end{lemma}

\begin{lemma}
	\label{le:L2Bijective}
	For every $A \in\LLub^*_2$, there exists a spoke-preserving function $\sqsp \in (\LLb_2)^{K_A}_\bullet$, and a bijective $\sqsp$-lift $B \in\LLub^*_2$ of $A$. 
\end{lemma}

Proceeding exactly as in the proof of  Lemma \ref{le:C2MAmalgamation}, and applying Lemma \ref{le:L2Bijective},  we  get:

\begin{lemma}
	\label{le:L2MAmalgamation}
	Let  $A, {B_i} \in\LLub^*_2$ be bijective, and let $\sqsp_i \in (\LLub^*_2)^{A}_{B_i}$, $i=0,1$ be such that $B_0^{\sqsb_0}=B_1^{\sqsb_1}$. There exist a bijective  $C  \in\LLub^*_2$, and surjective functions $\sqsp'_i \in (\LLub_2)^{B_i}_{C}$, $i=0,1$, that m-amalgamate $B_0$, $B_1$ over $A$. In particular, $\LLub_2^*$ has weak m-amalgamation.
\end{lemma}

%

Now we turn to the category $\LLub$. 

\begin{proposition}
	\label{pr:FAnDirected}
	$\LLub$ is directed.
\end{proposition}

\begin{proof}
	Let $A,B \in \LLub$. We can assume that $K_A \cap K_B=\{\rt\}$. Then $C=A \cup B$ is as required.  
\end{proof}

\begin{lemma}
\label{le:NiceFan}
For every fan $F$, there exist a unique regular fan $F' \supseteq F$,  and a unique contractive function $p_F \in \LLb^F_{F'}$ so that
\begin{enumerate}
\item the size of spokes in $*F'$ is equal to the size of the largest spoke in $F$,
\item for every $f \in F$ there are a unique spoke $\sigma_f \in *F'$, and $m_f \in \sigma_f$ such that $$p_F(\sigma_f)=p_F([\rt,m_f])=[\rt,f],$$
\item $p_F$ is injective on $[\rt,m_f]$.
\end{enumerate}
\end{lemma}

\begin{proof}
For every $f\in F$, extend a copy of $[\rt,f]$ to the size of the largest spoke of $F$, and let $p_F$ collapse the added terminal segment to $f$. These requirements uniquely determine $F'$ and $p_F$.
\end{proof}

\begin{lemma}
	\label{le:Lift}
	Let $F,G \in \LLb$. The fans $F' \supseteq F$, $G' \supseteq G$, and contractive functions $p_F \in \LLb^F_{F'}$, $p_G \in \LLb^G_{G'}$ as in Lemma \ref{le:NiceFan} are such that for every $\sqsp \in \LLb^F_G$, there exists a spoke-preserving $\sqsp'  \in \LLb^{F'}_{G'}$ such that $$\sqsp \circ p_G=p_F \circ \sqsp'.$$ 
\end{lemma}

\begin{proof}
Fix  $\sqsp \in \LLb^F_G$. 
Since $\sqsp$ is co-bijective, for every $g \in G$ there is a unique maximal $f \in F$ such that $[\rt,g]^\sqsb=[\rt,f]$. Define $\sqsp' \uhr \sigma_f \times \sigma_g$ so that on $[\rt,m_f] \times [\rt,m_g]$ it is a copy of $\sqsp \uhr [\rt,f] \times [\rt,g]$, and on the added terminal segments it is a spoke-monotone function from $\sigma_g$ onto the added terminal segment of $\sigma_f$. This can always be done after possibly replacing $G'$ by a longer regular extension.

Clearly, $\sqsp'$ is co-bijective and spoke-preserving, and $$\sqsp \circ p_G=p_F \circ \sqsp'.$$ 
	
	
\end{proof}

\begin{corollary}
	\label{co:PreserveSpokes}
	Let $F,G \in \LLb$, $A,B \in \MMub$ be such that $K_{A}=F$, $K_{B}=G$, and let $\sqsp \in \MMub^A_B$. For the fans $F' \supseteq F$, $G' \supseteq G$ as in Lemma \ref{le:NiceFan}, and the spoke-preserving morphism $\sqsp' \in \LLb^{F'}_{G'}$ as in Lemma \ref{le:Lift}, there exist a $p_F$-lift $A'$ of $A$, and a $p_G$-lift $B'$ of $B$ such that $B'$ is a $\sqsp'$-lift of $A'$.
\end{corollary}

\begin{proof}
Write $A=\sqsp_0\circ\sqsb_1$
and choose a corresponding presentation of $B$ witnessing that
$\sqsp$ is a morphism. Apply Lemma~\ref{le:Lift} simultaneously to
the morphisms in this finite diagram, choosing common sufficiently
long regular extensions. Define $A'=\sqsp'_0\circ\sqsb'_1$, and define $B'$ analogously. The commuting identities from
Lemma~\ref{le:Lift} show that $p_F$ and $p_G$ witness the first two
lifting properties and that $\sqsp'$ witnesses the last one.
\end{proof}

\begin{corollary}
	\label{co:BranchesonBanchesCoinitial}
	The class $\MMub \subseteq \LLub$ consisting of all $A \in\LLub$ such that
	\begin{enumerate}
		\item $K_A$ is regular, 
		\item $\sigma^A \cap \tau$ is a spoke or the root, for any $\sigma, \tau \in *K_A$
	\end{enumerate}
	is co-initial in $\LLub$.
\end{corollary}


\begin{proof}
	Let $A\in\LLub$, put $F=K_A$, and choose $G\in\LLb$ and
	$\sqsp_0,\sqsp_1\in\LLb^F_G$ such that $A=\sqsp_0\circ\sqsb_1$.
	Let $F'$, $G'$, $p_F$ and the spoke-preserving morphisms $\sqsp'_0,\sqsp'_1\in\LLb^{F'}_{G'}$ be given by Lemmas~\ref{le:NiceFan} and~\ref{le:Lift}. Then $A'=\sqsp'_0\circ\sqsb'_1$ is a $p_F$-lift of $A$, so it remains to verify that $A'\in\MMub$.
	Clearly, $K_{A'}=F'$ is regular. Let $f\in F'$ be an end and suppose that $f\sqsp'_1 g$. If $g'$ is the unique end of the spoke containing $g$, then $f\sqsp'_1 g'$. Moreover, by co-surjectivity of $\sqsp'_0$, there is an end $f'\in F'$ such that $f'\sqsp'_0 g'$. Hence $f'A'f$. It follows that whenever $\sigma^{A'}\cap\tau$ contains a non-root point, it also contains the end of $\tau$, and therefore, by monotonicity, $\sigma^{A'}\cap\tau=\tau$.
\end{proof}

The same argument as in the proof of the corresponding result for the Cantor fan, Proposition~\ref{pr:CatMCoInitial}, gives:

\begin{proposition}
	The full subcategory $\MMub \subseteq \LLub$ of all $A \in \LLub$ that satisfy the following conditions:
	\begin{enumerate}
		\item $K_{A}$ is regular,
		\item $A_\sigma \in \LLub_2^*$, for every $\sigma \in *K_A$,
		\item the spoke-graph of $A$ is in $\SSb$.
	\end{enumerate}
	is co-initial in $\LLub$.
\end{proposition}
Also, arguing as in the proof of Lemma \ref{le:C_IDense}, and using Proposition \ref{co:G_I*Dense}  instead of Proposition \ref{co:G_IDense}, we get

\begin{lemma}
	\label{le:L_IDense}
	$G_\LLb$ is a dense subset of $\Homeo(L)$.
\end{lemma}

Now we are ready to prove weak m-amalgamation in $\MMub$. The structure of the argument is almost exactly the same as the proof of weak m-amalgamation in $\DDub$, i.e., Lemma \ref{le:DAmalg}. The crucial modification is that on the way we need to apply Corollary \ref{co:PreserveSpokes} to reduce the construction to spoke-preserving digraphs.
 
\begin{theorem}
	$\LLub$  has weak m-amalgamation. In particular,  $\Homeo(L)$ has a comeager conjugacy class.
\end{theorem}

\begin{proof}
Let $A \in \MMub$. By Lemma \ref{le:L2MAmalgamation}, for every $\sigma \in *K_A$, we can find $B_\sigma \in \LLub^*_2$, and a spoke-preserving function $\nil_\sigma \in (\LLub^*_2)^{A_\sigma}_{B_\sigma}$ that is a weak m-amalgamation base for $A_\sigma$. As $\IIb \cap \GGRFb$ has amalgamation, using Lemma \ref{le:L2Lift}, and proceeding as in the beginning of the proof of Lemma \ref{le:DAmalg}, we can find $B \in \MMub$, and a spoke-preserving function $\nil \in \MMub^A_B$ such that $K_B$ is the union of $K_{B_\sigma}$, and $\nil$ is the union of $\nil_\sigma$, $\sigma \in *K_A$.
	
Fix $C_i\in \MMub$ and  $\sqsp_i \in \MMub^B_{C_i}$, $i=0,1$. Applying Corollary~\ref{co:PreserveSpokes} simultaneously to $\sqsp_0$ and $\sqsp_1$, using the same lift of the common source $B$, and keeping the same notation for the resulting objects and morphisms, we may assume that $$K_{C_0}^{\sqsb_0}=K_{C_1}^{\sqsb_1},$$ and that $C_0$, $C_1$ are spoke-preserving digraphs. Every two-spoke restriction of the lifted $B$ is a further lift of the corresponding $B_\sigma$, and hence remains a weak
$m$-amalgamation base over $A_\sigma$. Moreover, by the commuting
identities in Corollary~\ref{co:PreserveSpokes}, an $m$-amalgam of
the lifted diagrams gives an $m$-amalgam of the original ones after
composition with the lifting morphisms.

Let $S,T_i \in \SSb$ be the spoke-graphs of $B,C_i$, and let $\alpha_i \in \SSb^S_{T_i}$ be the morphisms induced by $\sqsp_i$.  As $\SSb$ has amalgamation, there are $U \in \SSb$, and morphisms $\alpha'_i \in  \SSb^{T_i}_U$, $i=0,1$, that amalgamate $T_0, T_1$ over $S$.
	
Fix an arrow $u \la v$ in $U$.  
Let $\sigma_i,\tau_i\in T_i$ be the images of $u,v$ under $\alpha'_i$, and let
$\sigma_B,\tau_B\in S$ be their common images under $\alpha_i$. 
Using Corollary \ref{le:L2MAmalgamation}, i.e., weak m-amalgamation in $\LLub^*_2$, we can fix paths $\sigma^{u,v}$, $\tau^{u,v}$ with corresponding amalgamating relations $D_{u,v}$ on $\sigma^{u,v} \times \tau^{u,v}$, and morphisms $$\sqsp^{u,v}_{i,\sigma} \subseteq \sigma_i \times \sigma^{u,v}, \quad \sqsp^{u,v}_{i,\tau} \subseteq \tau_i \times \tau^{u,v},$$ $i=0,1$, that are functions, and weakly m-amalgamate $C_i \uhr \sigma_i \times \tau_i$ over the corresponding restriction of $A$.  
As $\sqsp^{u,v}{i,\sigma}$ and $\sqsp^{u,v}{i,\tau}$ are functions, the construction of the amalgam $D$ from the $D_{u,v}$, and of the amalgamating morphisms $\sqsp'_i$ from $\sqsp^{u,v}_{i,\sigma}$ and $\sqsp^{u,v}_{i,\tau}$, proceeds exactly as in the last part of the proof of Lemma~\ref{le:DAmalg}; in particular, it uses amalgamation in $\IIb\cap\GGRFb$ to synchronize the local constructions on common spokes. The local $m$-amalgamation properties are preserved throughout this construction, and the same argument as in Lemma~\ref{le:DAmalg} shows that the intersection of the resulting composite morphisms is co-surjective. Thus $D$ weakly $m$-amalgamates $C_0$ and $C_1$ over $A$.
\end{proof}

Proposition \ref{pr:TopEntropy} immediately gives that the class $\ob(\DDub_0)$ is co-initial in $\ob(\MMub)$. Hence, exactly the same argument as for the Cantor fan yields:

\begin{theorem}
A generic homeomorphism of $L$ has no Li--Yorke pairs. In particular, zero topological entropy is generic for homeomorphisms of $L$.
\end{theorem}


\end{document}